\documentclass[11pt, a4paper]{article}
\usepackage{amsmath, amssymb, amsfonts, amsthm}
\usepackage[pdftex]{graphicx, color}

\newcommand{\E}{{\bf{E}}}
\newcommand{\PP}{{\bf{P}}}

\newtheorem{tm}{Theorem}
\newtheorem{lem}{Lemma}

\begin{document}

\parindent=0pt

\smallskip
\par\vskip 3.5em
\centerline{\Large \bf k-connectivity threshold  for superpositions}
\centerline{\Large \bf   of Bernoulli random graphs}

\vglue2truecm

\vglue1truecm
\centerline {Daumilas Ardickas, Mindaugas Bloznelis, Rimantas Vaicekauskas}

\bigskip

\centerline{Institute of Computer Science, Vilnius University}
\centerline{
\, \ Didlaukio 47, LT-08303 Vilnius, Lithuania} 

\vglue2truecm

\abstract{Let $G_1,\dots, G_m$ be independent identically distributed
Bernoulli random subgraphs of the complete graph ${\cal K}_n$ having
vertex sets of random sizes $X_1,\dots, X_m\in \{0,1,2,\dots\}$  and random edge densities $Q_1,\dots, Q_m\in [0,1]$. 
Assuming that each $G_i$ has a vertex of degree $1$ with positive probability, we establish the $k$-connectivity threshold as $n,m\to+\infty$ for the union 
$\cup_{i=1}^mG_i$ defined on the vertex set of ${\cal K}_n$. 

\smallskip
\par\vskip 3.5em
\section{Introduction}

Connectivity is a basic graph property. The strength of connectedness was  adressed in the fundamental paper by Menger \cite{Menger_1927} as early as 1927.
 Connectivity strenght of a random graph, where a given number of edges is inserted uniformly at random, has been studied in the seminal papers  
 by Erd\H{o}s and R\'enyi \cite{ErdosRenyi1959,ErdosRenyi1961}.
We recall that 
a graph is called $k$ vertex (edge) connected if the removal of any  $k-1$ vertices (edges) does not disconnect it.
Let $p_{n,m,k}$ denote the probability that  the Erd\H{o}s-R\'enyi random  graph on $n$ vertices with $m$ randomly inserted edges is $k$ (vertex) connected. 
It follows from the results of  \cite{ErdosRenyi1961}
 that as $m,n\to+\infty$  the probability $p_{n,m,k}$ undergoes a fast growth in the range 
$m=\frac{n}{2}\left(\ln n+(k-1)\ln\ln n +c_n\right)$.
For $c_n\to\pm\infty$ the probability $p_{n,m,k}\to0.5\pm0.5$; for $c_n\to c$ the probability
$p_{n,m,k}\to \exp\{-e^{-c}/(k-1)!\}$. 
In the literature this phenomenon is referred to as $k$-connectivity threshold.
It is important to mention that for $c_n\to-\infty$  the minimal degree of the  Erd\H os-R\'enyi random graph is at most $k-1$ with probability tending to $1$
\cite{ErdosRenyi1961} (this in turn implies 
$p_{n,m,k}=o(1)$).
In the subsequent literature  the $k$-connectivity property has been studied for various  random graph models: regular graphs \cite{Bollobas_1981}, \cite{Wormald_1981},  geometric graphs  \cite{Penrose_1999},
inhomogeneous binomial graphs  \cite{Devroye_Fraiman_2014},
\cite{Shang_2023}, random intersection graphs \cite{Zhao_yagan_Gligor_2017}, \cite{Bloznelis_Rybarczyk_2014}, see also  \cite{FriezeKaronski},
\cite{JansonLuczakRucinski2001}, \cite{Hofstad2017} and references therein.

In the present paper we establish the $k$-connectivity threshod for superpositions of Bernoulli random graphs.  We now   introduce the random graph model in detail.
Let $(X,Q), (X_1,Q_1)$,$\dots$, $(X_m,Q_m),\dots$ be a sequence of 
independent and identically distributed bivariate random variables taking 
values in $\{0,1,2\dots\}\times[0,1]$.
Given $n$ and $m$, let  
$G_1=({\cal V}_1,{\cal E}_1),\dots, G_m=({\cal V}_m,{\cal E}_m)$ 
be independent Bernoulli random subgraphs of the complete 
graph ${\cal K}_n$ having random vertex sets 
${\cal V}_i\subset {\cal V}$ and random
edge sets ${\cal E}_i$.   Here
${\cal V}=\{1,\dots, n\}=:[n]$ denotes the vertex set of ${\cal K}_n$.
% vertices having random vertex sets ${\cal V}_i$ and edge sets ${\cal E}_i$. 
 Each $G_i=({\cal V}_i,{\cal E}_i)$ is obtained by  firstly sampling 
 $(X_i,Q_i)$ and secondly by  selecting 
a 
%random 
subset of vertices ${\cal V}_i\subset {\cal V}$  of size 
$|{\cal V}_i|=\min\{X_i,n\}$ 
uniformly at random from the class of subsets of ${\cal V}$ 
of size $\min\{X_i,n\}$
and retaining edges between selected vertices independently at 
random with probability $Q_i$.  In particular, $G_i$  is a random 
graph on $\min\{X_i,n\}$ vertices, where every pair of vertices is 
linked by an edge independently at random with probability $Q_i$.
Note that given $i$ random variables $X_i$ and $Q_i$ do not need to be independent.
 We study   the union graph $G_{[n,m]}=({\cal V}, {\cal E})$ with  
the vertex set ${\cal V}=[n]$ and the edge set 
${\cal E}=\cup_{i\in[m]}{\cal E}_{i}$.  

Our motivation for studying
the random graph $G_{[n,m]}$ is based on two observations. Firstly,
$G_{[n,m]}$ is a natural generalisation of the Erd\H{o}s-R\'enyi graph
 $G[n,m]$, the random graph  on $n$ vertices  with $m$ randomly inserted edges \cite{ErdosRenyi1959}. Indeed, similarly to $G[n,m]$ our graph 
 $G_{[n,m]}$ is obtained by  inserting $m$ 
 bunches of edges ${\cal E}_1,\dots, {\cal E}_m$ (instead of  $m$ single edges), which, in addition, may overlap.
 Secondly, $G_{[n,m]}$ represents a null model of the community 
 affiliation graph of \cite{Yang_Leskovec2012,Yang_Leskovec2014}
that has attracted considerable attention in the literature.
 Community affiliation graph is a union of independent Bernoulli random graphs (communities), where the overlaps of the vertex sets of contributing communities are arranged by design.
Therefore, 
%it makes sense to  
 $G_{[n,m]}$ represents a random network of
% possibly 
overlapping communities $G_1,\dots, G_m$.
We also mention  related  random graph models of overlapping community networks  \cite{Vadon_Komjathy_Hofstad_2021} and \cite{Petti_Vempala_2022}.

In the parametric regime 
$m=\Theta(n)$ as $n,m\to+\infty$ the random graph $G_{[n,m]}$ 
admits an asymptotic degree distribution 
%(including power law) 
and
 non-vanishing global clustering coefficient
% \cite{Joona_Lasse_Mindaugas2021}, 
\cite{Lasse_Mindaugas2019}.
 The clustering property of $G_{[n,m]}$ indicates the abundance of small dense subgraphs. Asymptotic distributions of respective subgraph counts  are studied  in \cite{Joona_Lasse_Mindaugas2024}. 
 The effect of clustering on the component structure and percolation was studied  in \cite{Lasse_Mindaugas2019}. 
 Letting $m/n\to+\infty$ at the rate $m=\Theta(n\ln n)$ one can   make $G_{[n,m]}$ 
 connected with a high probability. The connectivity threshold  (under various conditions on the distribution of $(X,Q)$) was studied in \cite{Daumilas_Mindaugas2023, BergmanLeskela}, 
\cite{Dominykas_Mindaugas_Rimantas2023}, 
\cite{GodehardJaworskiRybarczyk2007}.

%In the present paper we establish the $k$-connectivity threshold for 
%$G_{[n,m]}$ and $k\ge 2$. 
%We recall that a graph is called $k$ vertex (edge) connected if removal of
% any $k-1$ vertices (edges) does not make the (remaining) graph 
% disconnected.  

Before formulating our result, we introduce some notation.
 Given integer $x\ge 0$ and number $q\in [0,1]$ we denote by $G(x,q)$  the Bernoulli random graph with the vertex set
  $[x]=\{1,\dots, x\}$ and  with the edge probability $q$ (any pair of vertices is declared adjacent independently at random with probability $q$). 
 $G(0,q)$ refers to the empty graph having no vertices. 
 We denote
\begin{displaymath}
h(x,q)=1-(1-q)^{(x-1)_{+}}
\end{displaymath}
the probability that vertex $1$ is not isolated in $G(x,q)$.
 We write for short 
 $(x)_+=\max\{x,0\}$ and assign value $1$ to the expression  $0^0$. Note that $h(1,q)=h(0,q)=0$ for any $q\in[0,1]$.
  We denote by ${\mathbb I}_{\cal A}$ the indicator function of an event (set) ${\cal A}$.  We denote by ${\bar {\cal A}}$ the event complement to ${\cal A}$.
Furthermore, we denote
\begin{align}
\nonumber
&
\alpha
=
\E(Q{\mathbb I}_{\{X\ge 2\}}),
\qquad
\kappa^*
=
\E (Xh(X,Q)),
\qquad
\tau^*
=
\E\bigl((X)_2Q(1-Q)^{X-2}\bigr),
\\
\label{lambda*}
&
\lambda^*
=
\lambda_{n,m,k}^*
=
\ln n+(k-1)\ln \frac{m}{n}-\frac{m}{n}{\kappa^*}.
\end{align}

\begin{tm}\label{Theorem_1} Let $k\ge 2$ be an integer.
Let $n\to+\infty$. 
Assume that $m=m(n)=\Theta(n\ln n)$. 
Assume that  
$\alpha>0$ and $\tau^*>0$ and
\begin{align}
\label{2024-02-23}
&
\E \left(Xh(X,Q)\ln(1+X)\right)<\infty, 
\\
\label{2024-02-23+1}
&
\E\left(X\min\{1,XQ\}\ln(1+X)\right)<\infty,
\\
&
\label{eta_j}
\E \left(X^{j}Q^{j-1}\right)<\infty,
\qquad
 \qquad 
 2\le j\le k.
 \end{align}
Then 
\begin{align}
\label{Theorem_1_1}
&
\PP\{G_{[n,m]}{\rm{\ is \ vertex}}\ k{\rm{-connected}}\}
\to 1 
\
\
\
{\rm{for}}
\
\
 \
\lambda_{n,m,k}^*\to-\infty,
\\
\label{Theorem_1_2}
&
\PP\{G_{[n,m]}{\rm{\ is \ edge}}\ k{\rm{-connected}}\}
\
\
\to 0
\
\
\
 {\rm{for}}
\
\
\
\lambda_{n,m,k}^*\to+\infty.
\end{align}
\end{tm}

We remark that since the vertex $k$-connectivity implies edge 
$k$-connectivity, the dichotomy (\ref{Theorem_1_1}), 
(\ref{Theorem_1_2}) extends to either sort of $k$-connectivity  (edge and vertex connectivity).

We  comment on the conditions of Theorem \ref{Theorem_1}. 
Condition $\alpha>0$ excludes the trivial case where $G_{[n,m]}$ is empty. Indeed, 
 $\alpha=0$ implies $\PP\{G_{[n,m]}$ has no edges$\}=1$.
 Furthermore, condition $\tau^*>0$ excludes the case,
 where each  community $G_{i}$ is either empty or is a clique of (random) size of at least $3$. It allso implies that $G_i$ has a vertex of degree $1$ with positive probability.  
 The condition $\tau^*>0$ plays important role in the proof of 
 (\ref{Theorem_1_2}), where
 we show that $G_{[n,m]}$ has a vertex of degree   at most $k-1$ with high probability.
In particular, under the assumption that $\tau^*>0$ the $k$-connectivity threshold for community affiliation graph with randomly assigned community memberships follows a pattern similar  to that of the  Erd\H{o}s-R\'enyi random graph described in the seminal paper
\cite{ErdosRenyi1959}:  an obstacle to $k$-connectivity is a vertex of degree at most $k-1$.
The case where $\tau^*=0$ (i.e., the case of clique communities of sizes $\ge 3$) needs a different approach and will be considered elsewhere.
 Conditions (\ref{2024-02-23+1}),
  (\ref{eta_j}) are technical and can probably be relaxed.  
 %We note that  conditions $\E (X^2Q)<\infty$  and  $\alpha>0$  ensure that  
  
  We note that 
 for $\E (X^2Q)<\infty$   
and  $\alpha>0$
the  parameter
$\kappa^*$  (that enters $\lambda^*_{n,m,k}$) is well defined  and it is bounded away from zero. Indeed, we show in Fact 3 (Section 2 below) that 
%. This follows from the inequalities 
\begin{equation}
\label{2024-01-04}
\E((X)_2Q)
\ge
\kappa^*
\ge
\alpha.
\end{equation} 
We also discuss condition  $m=\Theta(n\ln n)$. We impose this %technical 
condition to  exclude sequences $m=m(n)$ 
satisfying  $\lambda_{n,m,k}^*\to-\infty$ and such that 
$m(n)=o(n)$
because for such sequences (\ref{Theorem_1_1})  may fail.  For example, 
for bounded $X$ and for $m=o(n)$  the number of isolated vertices is at least $n-\sum_{i=1}^m|{\cal V}_i|\ge n-\sum_{i=1}^mX_i
=n-O_P(m)=O_P(n)$.  An example of such a sequence is  $m=m(n)=\lfloor n^{\frac{k-2}{k-1}}{ \ln^{-1} n}\rfloor$. 

Finally, we note that
the validity of (\ref{Theorem_1_1}) extends 
to the case where $n\ln n=o(m)$. Similarly,  the validity of (\ref{Theorem_1_2}) extends 
to the case $m=o(n\ln n)$. 
 To see this we combine Theorem \ref{Theorem_1} with the 
 coupling argument:
given two sequences $m_{-}=m_{-}(n)$ and
$m_{+}=m_{+}(n)$ such that $m_{-}\le m_{+}$ there is a natural coupling of  $G_{[n,m_{-}]}$ and $G_{[n,m_{+}]}$ such that
 $\PP\{ G_{[n,m_{-}]}\subset G_{[n,m_{+}]}\}=1$ (we obtain $G_{[n,m_{-}]}$ from  $G_{[n,m_{+}]}$ by removing $m_{+}-m_{-}$ layers). Clearly,   $k$-connectivity of $G_{[n,m_{-}]}$
implies  $k$-connectivity of $G_{[n,m_{+}]}$. Similarly, if 
$G_{[n,m_{+}]}$ fails to be $k$-connected then so does $G_{[n,m_{-}]}$.}

\bigskip

\section{Proof of Theorem \ref{Theorem_1}}

The proof of Theorem \ref{Theorem_1} consists of two parts. We prove 
 (\ref{Theorem_1_2}) in Lemma \ref{degree} and (\ref{Theorem_1_1})
 in Lemma \ref{Lemma_1}.  
 For convenience we first formulate
  Lemmas \ref{degree} and \ref{Lemma_1}. Then  we prove Theorem \ref{Theorem_1}. Proofs of Lemmas \ref{degree} and \ref{Lemma_1}
  are postponed until after the proof of Theorem \ref{Theorem_1}.
We begin with notation and auxiliary facts.

{\it Notation.}
We write, for short, 
$G=G_{[n,m]}$.
Given $S\subset {\cal V}=[n]$, 
we denote by  $G^{(-S)}$ the subgraph
 of $G$ induced by the vertex 
 set ${\cal V}\setminus S$. 
 For integer $k\ge 2$, introduce event 
 \begin{displaymath}
 {\cal B}_k=\Bigl\{ \exists  S\subset {\cal V}:  |S|\le k-1,  
G^{(-S)}\,
{\text{has  a  component  on}}
\
r
\
{\text{vertices  for  some}}
\
2\le r\le \frac{n-|S|}{2}
\Bigr\}.
\end{displaymath}
We denote by $d(v)$ (respectively $d_i(v)$) the 
degree of $v\in {\cal V}$ in $G$ 
(respectively $G_i$). 
We put $d_i(v)=0$ for $v\notin {\cal V}_i$.
Let 
\begin{displaymath}
N_t
=
\sum_{v\in{\cal V}}{\mathbb I}_{\{d(v)=t\}}
\end{displaymath}
be the number of vertices of degree $t$ in $G$.
 We write  ${\tilde X}=\min\{X,n\}$ and
  ${\tilde X}_i=\min\{X_i,n\}$, for $1\le i\le m$, and
denote  $\eta_j={\bf{E}} (X^jQ^{j-1})$,
$j=2,3,\dots$,
\begin{align}
\nonumber
&
\kappa
=
\kappa_n
=
\E ({\tilde X}h({\tilde X},Q)),
\qquad
\tau
=
\tau_n
=
\E
\bigl(
({\tilde X})_2Q(1-Q)^{{\tilde X}-2}\bigr),
\\
&
\mu=\E \left(Xh(X,Q)\ln(1+X)\right), 
\qquad
\mu'=\E\left(X\min\{1,XQ\}\ln(1+X)\right),
\\
\label{lambda}
&
\lambda_{n,m,k}
=
\ln n+(k-1)\ln \frac{m}{n}-\frac{m}{n}{\kappa}.
\end{align}
We note that quantities $\kappa$ and $\tau$ depend on $n$ and tend to 
$\kappa^*$ and $\tau^*$  as $n\to \infty$.

For a sequence of random variables $\{\zeta_n, n\ge 1\}$ we write 
$\zeta_n=o_P(1)$ if $\PP\{|\zeta_n|>\varepsilon\}=o(1)$ as $n\to+\infty$ for any $\varepsilon>0$. We write $\zeta_n=O_P(1)$ if 
$\lim_{C\to+\infty}\sup_n\PP\{|\zeta_n|>C\}=0$.

{\bf Fact 1}. 
%Remark \label{kappa-mu}
Assume that $\mu<\infty$. Then 
$0\le \kappa^*-\kappa\le \frac{\mu}{\ln(1+n)}$. Moreover, for
$m=\Theta(n\ln n)$ we have 
$0\le \lambda_{n,m,k}-\lambda^*_{n,m,k}=o(1)$ as 
$n \to+\infty$.

{\it Proof of Fact 1.}
Denote $\mu_n={\bf{E}}\left(Xh(X,Q)\ln(1+X){\bf I}_{\{X>n\}}\right)$.
We have 
\begin{eqnarray}
\nonumber
\kappa^*-\kappa
&=&
{\bf{E}}\left(Xh(X,Q){\bf I}_{\{X>n\}}\right)
-
{\bf{E}}\left(nh(n,Q){\bf I}_{\{X>n\}}\right)
\\
\nonumber
&\le& 
{\bf{E}}\left(Xh(X,Q){\bf I}_{\{X>n\}}\right)
\le
\frac{\mu_n}{\ln(1+n)}.
\end{eqnarray}
Note that the right side of the first identity is non-negative because  $x\to h(x,q)$ is nondecreasing. Hence  $0\le \kappa^*-\kappa$.
Furthermore, the  inequality $\mu_n\le \mu$
implies $\kappa^*-\kappa\le \frac{\mu}{\ln(1+n)}$.
Moreover,  $\mu<\infty$ implies $\mu_n=o(1)$. Hence 
\begin{eqnarray}
\nonumber
\lambda_{n,m,k}-\lambda_{n,m,k}^*
&=&
\frac{m}{n}(\kappa^*-\kappa)
\le 
\frac{m}{n}
\frac{\mu_n}{\ln(1+n)}
=
o\left(\frac{m}{n}
\frac{1}{\ln(1+n)}
\right)
=o(1).
\quad
\qed
\end{eqnarray}

{\bf Fact 2}.
%Remark \label{tau-tau}
 We have
$|\tau-\tau^*|\le {\bf{E}} \left((X)_2Q{\bf I}_{\{X>n\}}\right)$. Consequently, 
$\eta_2<\infty$ implies $|\tau-\tau^*|=o(1)$
as $n\to+\infty$.

{\it Proof of Fact 2.}
We obtain $|\tau-\tau^*|\le {\bf{E}} \left((X)_2Q{\bf I}_{\{X>n\}}\right)$  from the inequalities
\begin{eqnarray}
\nonumber
\tau^*-\tau
&=&
{\bf{E}}\left((X)_2Q(1-Q)^{X-2}{\bf I}_{\{X>n\}}\right)
-
{\bf{E}}\left((n)_2Q(1-Q)^{n-2}{\bf I}_{\{X>n\}}\right)
\\
\nonumber
&\le&
{\bf{E}}\left((X)_2Q(1-Q)^{X-2}{\bf I}_{\{X>n\}}\right)
\le
{\bf{E}}\left((X)_2Q{\bf I}_{\{X>n\}}\right),
\\
\nonumber
\qquad
\tau-\tau^*
&=&
{\bf{E}}\left((n)_2Q(1-Q)^{n-2}{\bf I}_{\{X>n\}}\right)
-
{\bf{E}}\left((X)_2Q(1-Q)^{X-2}{\bf I}_{\{X>n\}}\right)
\\
\nonumber
&\le&
{\bf{E}}\left((n)_2Q(1-Q)^{n-2}{\bf I}_{\{X>n\}}\right)
\le
{\bf{E}}\left((X)_2Q{\bf I}_{\{X>n\}}\right).
\quad
\qquad
\qquad
\qed
\end{eqnarray}

{\bf Fact 3}. Assume that $\E((X)_2Q)<\infty$. Then 
(\ref{2024-01-04}) holds.

{\it Proof of Fact 3.}
 For integer $x\ge 2$ and $0\le q\le 1$ we have
\begin{displaymath}
q(x-1)
=
q\sum_{i=0}^{x-2}1
\ge 
q\sum_{i=0}^{x-2}(1-q)^i=h(x,q)\ge h(2,q)
=
q.
 \end{displaymath} 
Now inequalities $q(x-1)\ge h(x,q)\ge q$ for $x\ge 2$ imply
$(x)_2q\ge xh(x,q)\ge xq\ge q$. Consequently, we have
 \begin{displaymath}
\E ((X)_2Q{\mathbb I}_{\{X\ge 2\}})
\ge 
\E (Xh(X,Q){\mathbb I}_{\{X\ge 2\}})
 \ge 
  \E (Q{\mathbb I}_{\{X\ge 2\}}).
 \end{displaymath}
 Invoking identities
 $\E ((X)_2Q)=\E ((X)_2Q{\mathbb I}_{\{X\ge 2\}})$ and 
 $\E (Xh(X,Q))$$=$
 $\E (Xh(X,Q){\mathbb I}_{\{X\ge 2\}})$ (the latter one follows from the identities
$h(1,q)=h(0,q)=0$) we obtain
 (\ref{2024-01-04}) . 
$\qed$

\begin{lem}\label{degree}Let $k\ge 2$ be an integer. Let $n,m\to+\infty$. Assume that $m=m(n)=\Theta(n\ln n)$.
Assume that $\alpha>0$, $\tau^*>0$, $\eta_2<\infty$, and 
$\mu'<\infty$. For $k\ge 3$ we assume, in addition, that $\eta_3<\infty$. 
Then for $\lambda_{n,m,k}\to+\infty$
 we have $\PP\{ N_{k-1}\ge 1\}\to 1$.
Moreover, for $\lambda_{n,m,k}\to-\infty$
we have $N_t=o_P(1)$ for $t=0,1,\dots, k-1$.
\end{lem}

\begin{lem}
\label{Lemma_1} 
Let $k\ge 2$ be an integer.
Let $n\to+\infty$. 
Assume that 
$m=m(n)\to\infty$ and $m=\Theta(n\ln n)$.
Assume that $\alpha>0$, $\mu<\infty$ and and  $\eta_j<\infty$,
 for $2\le j\le k$.
Assume that $\lambda_{n,m,k}^*\to -\infty$ and 
$|\lambda_{n,m,k}^*|=o(\ln\ln n)$.
Then
\begin{equation}
\label{component2}
\PP\{{\cal B}_k\}=o(1).
\end{equation} 
\end{lem}
Proofs of Lemmas  \ref{degree} and \ref{Lemma_1} 
are postponed until after the proof of Theorem 
\ref{Theorem_1}.

\begin{proof}[Proof of Theorem \ref{Theorem_1}]
In view of Fact 1 
%\ref{kappa-mu} 
we have 
$\lambda_{n,m,k}\to\pm\infty\Leftrightarrow \lambda_{n,m,k}^*\to\pm \infty$. 

Let us show (\ref{Theorem_1_2}). 
By Lemma \ref{degree},  with probability tending to $1$ there exists a vertex of degree $k-1$. Hence by removing  $k-1$ edges one can make $G_{[n,m]}$ disconnected.

Let us show (\ref{Theorem_1_1}).
%Now assume that $\lambda_{n,m,k}^*\to-\infty$.
We claim that it suffices to show
(\ref{Theorem_1_1}) in the case where 
$|\lambda_{n,m,k}^*|=o(\ln\ln n)$.  To see why this is true
% fix a (large) $n$ and 
consider two sequences $m'=m'(n)$ and $m''=m''(n)$ such that
$m', m''=\Theta(n\ln n)$, $\lambda_{n,m',k}^*\to -\infty$,
$\lambda_{n,m'',k}^*\to -\infty$  and 
$|\lambda_{n,m',k}|=o(\ln\ln n)$. Put $m_1=m'\wedge m''$ and $m_2=m' \vee m''$.  
Analysis of the function $m\to \lambda_{n,m,k}^*$ defined in (\ref{lambda*}) shows that (for sufficiently large $n$) we have
$m_1\le m_2
\Leftrightarrow 
|\lambda_{n,m_1,k}^*|\le|\lambda_{n,m_2,k}^*|$.
Furthermore, the coupling argument (see discussion after Theorem 
\ref{Theorem_1}) implies
\begin{displaymath}
\PP\{G_{[n,m_1]}{\rm{\ is \ vertex}}\ k{\rm{-connected}}\}
\le
\PP\{G_{[n,m_2]}{\rm{\ is \ vertex}}\ k{\rm{-connected}}\}.
\end{displaymath}
Consequently, the asymptotic connectivity (\ref{Theorem_1_1}) of 
$G_{[n,m_1]}$ implies that of $G_{[n,m_2]}$. Hence it suffices to show (\ref{Theorem_1_1}) in the case where 
$|\lambda_{n,m,k}^*|=o(\ln\ln n)$.

For $\lambda_{n,m,k}^*\to-\infty$ Lemma \ref{degree} shows 
that with 
probability tending to $1$ there is no vertex of degree less than $k$. 
Furthermore, for $\lambda_{n,m,k}^*\to-\infty$ satisfying 
$|\lambda_{n,m,k}^*|=o(\ln\ln n)$,   Lemma \ref{Lemma_1} shows that  (with probability 
tending to $1$) removal of any (non-empty) 
set $S\subset {\cal V}=[n]$ of vertices of size at most $k-1$ does 
not create a component of size 
$r\in \left[2,\frac{n-|S|}{2}\right]$ (in  a subgraph of $G_{[n,m]}$ induced by the vertex 
set ${\cal V}\setminus S$). 
An immediate consequence of these two lemmas is 
that $G_{[n,m]}$ is vertex $k$-connected with probability 
tending to $1$.
\end{proof}

\subsection{Proof of Lemma \ref{Lemma_1}} 
Before the proof of Lemma \ref{Lemma_1}
we introduce some notation and establish
 auxiliary  results stated  Lemmas \ref{sr},  \ref{2023-12-05Lema} below.

%{\color{blue}

Let $A_x\subset [n]$ be a  subset  sampled uniformly at random from the class of subsets of size $x$. 
Let $q\in [0,1]$. 
Let $G_x=(A_x,{\cal E}_x)$ be Bernoulli random graph with the vertex set $A_x$ and with egde probability $q$. That is, 
given $A_x$, every pair 
$\{u,v\}\subset A_x$ is declared adjacent at random with probability $q$ independently of the other pairs.  Here ${\cal E}_x$ denotes the collection of  pairs $\{u,v\}$  of adjacent vertices.
For $i, j \in \mathbb{N}$  we denote the set 
$(i, j]_{\mathbb N} = \{i+1,i+2,...,j\}$.
For integers $r\ge 1$ and $s\ge 0$  introduce event
\begin{displaymath}
{\cal B}_{r,s}=\Bigl\{\forall \{u,v\}\in{\cal E}_x 
\text{ we have either } \{u,v\} \cap (s,s+r]_{\mathbb N} = \emptyset 
\text { or } \{u,v\} \cap (s+r,n]_{\mathbb N} = \emptyset \Bigr\}.
\end{displaymath}
${\cal B}_{r,s}$ means that none of the edges of $G_x$ connect subsets 
$(s,s+r]_{\mathbb N}$ and $(s+r,n]_{\mathbb N}$ of  ${\cal V}=[n]$.
We denote 
%\end{document}
\begin{displaymath}
q_{r,s}(x,q)=q_{r,s,n}(x,q)=\PP\{{\cal B}_{r,s}\}
\end{displaymath}
and put
$q_{r,s}(x,q)=1$
for
$x=0,1$.
Furthermore, we denote
\begin{displaymath}
{\hat q}_{r,s}=\E q_{r,s}({\tilde X},Q).
\end{displaymath}
It is easy to see that that
   \begin{displaymath}
   {\hat q}_{r,s}
   =
   \PP\{{\tilde X}\le 1\}+\E \bigl(
   q_{r,s}({\tilde X},Q){\mathbb I}_{\{{\tilde X}\ge 2\}}
   \bigr).
\end{displaymath}

%\end{document}

\begin{lem}\label{sr}
  For $s\ge 0$, $1\le r \le (n-s)/2$, $2\le x\le n$ and $0\le q\le 1$ we have
  \begin{align}
  \label{sr_1}
  q_{r,s}(x,q) 
  & 
  \le 
  1-2q\frac{r(n-s-r)}{(n-s)(n-s-1)}+\frac{(s+1)_2}{n},
  \\
  \label{sr_2}
  q_{r,s}(x,q) 
  & 
\le 
1
-
\left(1-e^{-\frac{rx}{n-s}}
-
R_{1,s} 
\right)
h(x,q)
 +
 \frac{srx^2}{n(n-s)}h(x,q).
  \end{align}
Here 
   $R_{1,s}=\frac{r^2}{(n-s-r)^2}$.
\end{lem}

\begin{proof}[Proof of Lemma \ref{sr}]
For $s=0$ inequalities 
\begin{align}
  \label{sr_1+}
  q_{r,0}(x,q) 
  & 
  \le 
  1-2q\frac{r(n-r)}{n(n-1)},
  \\
  \label{sr_2+}
  q_{r,0}(x,q) 
  & 
\le 
1
-
\left(1-e^{-\frac{rx}{n}}
-
\frac{r^2}{(n-r)^2}
\right)
h(x,q)
  \end{align}
  have been shown in Lemma 1 of \cite{Daumilas_Mindaugas2023}. 
  Here we show
  (\ref{sr_1}), (\ref{sr_2}) for $s\ge 1$.

Introduce hypergeometric random variable $H=|A_x\cap [s]|$ ($=$ the number of elements of $A_x$ that belong to $[s]=[1,\dots, s]$). 
In the proof we use inequalities (see, e.g. Lemma 6 in \cite{Bloznelis2013AAP})
\begin{align}
\label{2023-12-04}
\PP\{H\ge t\}\le (x)_t(s)_t/((n)_t t!), 
\qquad
t=1,2,\dots
\end{align}
and the identity (which follows by the total probability formula)
\begin{align}
\label{sr_p00}
\PP\{{\cal B}_{r,s}\}
=
\sum_{i=0}^{x\wedge s}
\PP\{{\cal B}_{r,s}|H=i\}
\PP\{H=i\}.
\end{align}

Proof of (\ref{sr_1}).
For $x-i\ge 2$  inequality (\ref{sr_1+})
%(6) of Lemma 1 of \cite{Daumilas_Mindaugas2023} 
implies
\begin{align}
\label{2025-03-19}
\PP\{{\cal B}_{r,s}|H=i\}
\le
1 
-2q\frac{r(n-s-r)}{(n-s)(n-s-1)}.
\end{align}
Here the right side upper bounds the probability that the subgraph of $G_x$ induced by the vertex set $A_x\cap (s,n]_{\mathbb N}$ has no edge connecting sets  $(s,s+r]_{\mathbb N}$ and $(s+r,n]_{\mathbb N}$.
Combining this inequality with (\ref{2023-12-04}), (\ref{sr_p00}) we obtain
\begin{align}
\nonumber
\PP\{{\cal B}_{r,s}\}
&
\le \sum_{0\le i\le x-2}
\PP\{{\cal B}_{r,s}|H=i\}
\PP\{H=i\}
+
\PP\{H\ge x-1\}
\\
\label{2024-01-04+1}
&
\le
\PP\{H\le x-2\}
\left(1 -2q\frac{r(n-s-r)}{(n-s)(n-s-1)}\right)
+
\frac{(s)_{x-1}(x)_{x-1}}{(n)_{x-1}(x-1)!}.
\end{align}
We note that the last term vanishes for $s<x-1$. For $s\ge x-1$ the last term
\begin{displaymath}
\frac{(s)_{x-1}(x)_{x-1}}{(n)_{x-1}(x-1)!}
=
\frac{(s)_{x-1}x}{(n)_{x-1}}\le \frac{sx}{n}
\le 
\frac{s(s+1)}{n}.
\end{displaymath}
Invoking this bound in (\ref{2024-01-04+1}) and 
using $\PP\{H\le x-2\}\le1$ we obtain (\ref{sr_1}).

Proof of (\ref{sr_2}).
  We write (\ref{sr_p00}) in the form,
\begin{align}
\label{sr_p0}
\PP\{{\cal B}_{r,s}\}
=
1+\sum_{i=0}^{x\wedge s} 
(\PP\{{\cal B}_{r,s}|H=i\}-1)\PP\{H=i\}
\end{align}
 and proceed similarly as in (\ref{2025-03-19}). For $x-i\ge 2$ we apply (\ref{sr_2+}) to upper bound the probabilities 
\begin{align}
\nonumber
\PP\{{\cal B}_{r,s}|H=i\}
\le
1 
-\left(1-e^{-\frac{r}{n-s}(x-i)}-R_{1,s}\right)h(x-i,q).
\end{align}
Note that this inequality holds for $x-i\le 1$ as well, because 
we have $\PP\{{\cal B}_{r,s}|H=i\}
\le
1$ and 
$h(1,q)=h(0,q)=0$.
Next, denoting
$\xi_i=(1-e^{-\frac{r}{n-s}(x-i)})h(x-i,q)$ and using  
$0\le h(x-i,q)\le h(x,q)$ we obtain
\begin{align}
\label{sr_p2}
\PP\{{\cal B}_{r,s}|H=i\}
\le 1-\xi_i+ R_{1,s}h(x,q).
\end{align}
Furthermore, invoking (\ref{sr_p2}) in (\ref{sr_p0}) 
we obtain
\begin{align}
\label{2023-12-02+10}
\PP\{{\cal B}_{r,s}\}
\le 
1-\sum_{i=0}^{x\wedge s} 
\xi_i\PP\{H=i\}
+
R_{1,s}h(x,q)
\le 1-\xi_0\PP\{H=0\}
+R_{1,s}h(x,q).
\end{align}
In the last inequality we used  $\xi_i\ge 0$. 
Next, we lower bound $\PP\{H=0\}$ using  (\ref{2023-12-04}),
\begin{displaymath}
\PP\{H=0\}=1-\PP\{H\ge 1\}\ge 1-\frac{sx}{n}.
\end{displaymath}
Invoking this inequality in  (\ref{2023-12-02+10})
we obtain
(\ref {sr_2}):
\begin{align*}
\PP\{{\cal B}_{r,s}\}
&
\le 
 1-\xi_0+\xi_0\frac{sx}{n}+R_{1,s}h(x,q)
 \\
 &
 \le
 1-\xi_0+\frac{rsx^2}{n(n-s)}h(x,q)+R_{1,s}h(x,q).
\end{align*}
In the last step we applied  inequality
$1-e^{-a}\le a$ to $a=\frac{r}{n-s}x$ and
estimated 
\begin{displaymath}
\xi_0= \left(1-e^{-\frac{r}{n-s}x}\right)h(x,q)
\le \frac{rx}{n-s}h(x,q).
\end{displaymath}
\end{proof}
 
 %} 

\begin{lem}\label{2023-12-05Lema} 
Let $s\ge 1$ be an integer.
Assume that $\mu<\infty$.
 %and denote  
%$\mu=\E\bigl(Xh(X,Q)\ln(1+X)\bigr)$.
For any $0<\beta<1$ there exists $n_*>0$ depending on $s$,
$\beta$ and  the probability distribution of 
$(X,Q)$ such that for 
$n>n_*$ we have  for  each $2\le r\le n^{\beta}$
\begin{align}
\label{2023-12-05+11}
{\hat q}_{r,s}
&
\le 
1-\frac{r}{n-s}\kappa^*+(2+\mu (s+1))\frac{r}{(n-s)\ln n}.
\end{align}
\end{lem}

\begin{proof}
%{\color{blue}
Proof of (\ref{2023-12-05+11}).
Fix $0<\beta<1$. 
 We write (\ref{sr_2}) in the form
\begin{align}
\label{2023-12-05+12}
 q_{r,s}(x,q) 
\le 
1
-
\left(\frac{rx}{n-s}-R_{1,s}-R_{2,s}(r,x) 
\right)
h(x,q)
 +
 \frac{srx^2}{n(n-s)}h(x,q),
  \end{align}
where 
   $R_{2,s}(r,x)=e^{-\frac{rx}{n-s}}-1+\frac{rx}{n-s}$. 
Then we plug 
   $({\tilde X},Q)$ 
   in (\ref{2023-12-05+12}) (in the place of $(x,q)$) and  take expected values of both sides.
We evaluate the expection of each term 
%   In doing so we treat every term 
(on the right of (\ref{2023-12-05+12})) separately.

We start with $\E\left(\frac{r}{n-s}{\tilde X}h({\tilde X},Q)\right)$. We have 
\begin{equation}
\label{2024-01-04+20}
\E\left(\frac{r}{n-s}{\tilde X}h({\tilde X},Q)\right)
=
\frac{r}{n-s}\kappa_n
\ge
\frac{r}{n-s}\left(\kappa^*-\frac{\mu}{\ln (1+n)}\right)
.
\end{equation}
In the last step we invoked inequality 
$\kappa_n\ge \kappa^*-\mu\ln^{-1}(1+n)$
of Fact 1. 
%\ref{kappa-mu}.
% , which follows by the chain of inequalities 
%\begin{align*}
%0\le \varkappa^*-\varkappa_n
%&
%\le
%\E\left(Xh(X,Q){\mathbb I}_{\{X>n\}}\right)
%\\
%&
%\le
%\frac{1}{\ln(1+n)}
%\E\left(Xh(X,Q)\ln(1+X){\mathbb I}_{\{X>n\}}\right)
%\le 
%\frac{\mu}{\ln(1+n)}.
%\end{align*}
%Here the  very first inequality follows from the fact that $x\to h(x,q)$ is nondecreasing.

Next, we note that  $h(x,q)\le 1$ implies 
$\E(R_{1,s}h({\tilde X},Q))
\le 
R_{1,s}$. For $r\le n^{\beta}$ we obtain 
\begin{equation}
\label{2024-01-04+21}
\E(R_{1,s}h({\tilde X},Q))
\le 
R_{1,s}
\le 
\frac{r n^{\beta}}{(n-n^{\beta}-s)^2}
\le \frac{r}{n\ln n},
\end{equation}
where the last inequality holds  for sufficiently large $n$.

Furthermore, we claim that there exists $n_0>0$ (depending on $s$, $\beta$ and the distribution of $(X,Q)$) such that  for each $r \le n^{\beta}$ we have
 \begin{align}
 \label{2023-12-05+10}
 \E 
 \bigl(R_{2,s}(r, {\tilde X})h({\tilde X},Q)
 \bigr)
 \le 
 \frac{r}{n\ln n}.
\end{align}  
This inequality is shown in  formula ({\color{red}26}) of  \cite{Daumilas_Mindaugas2023} for $s=0$. The same proof yields 
(\ref{2023-12-05+10}) for arbitrary, but fixed $s$.

Finally, we upper bound the expected value of the last term on the right of  (\ref{2023-12-05+12}).
We split 
\begin{displaymath}
I:=\E \bigl({\tilde X}^2h({\tilde X},Q)\bigr)
=
 \E \bigl({\tilde X}^2h({\tilde X},Q){\mathbb I}_{\{X\le n\}}\bigr) 
 +
  \E \bigl({\tilde X}^2h({\tilde X},Q){\mathbb I}_{\{X> n\}}\bigr) 
  =:I_1+I_2.
  \end{displaymath}
 Combining  identity $h(x,q)=0$, for $x\in\{0,1\}$, 
 with  inequality
 \begin{displaymath}
 x^2=x\ln(1+x)\frac{x}{\ln(1+x)}\le x\ln (1+x)\frac{n}{\ln (1+n)},
 \end{displaymath}
 which holds for $2\le x\le n$, we  upper bound 
 \begin{align*}
 I_1
 =
\E\bigl(X^2h(X,Q){\mathbb I}_{\{2\le  X\le n\}}\bigr)
\le
\frac{n}{\ln(1+n)}\E (X(\ln(1+X)) h(X,Q){\mathbb I}_{\{2\le  X\le n\}}\bigr).
\end{align*}
Next, using the fact that $x\to h(x,q)$ is nondecreasing and $n\ln (1+n)\le x\ln (1+x)$ for $x>n$ we upper bound 
\begin{align*}
I_2=
\E\bigl(n^2h(n,Q){\mathbb I}_{\{X> n\}}\bigr)
\le 
\frac{n}{\ln(1+n)}\E (Xh(X,Q)\ln(1+X){\mathbb I}_{\{X> n\}}\bigr).
\end{align*}
We conclude that $I=I_1+I_2\le
\frac{n}{\ln(1+n)}
\mu
$.
Hence
\begin{equation}
\label{2024-01-04+25}
\E\left( \frac{sr{\tilde X}^2}{n(n-s)}h({\tilde X},Q)
\right)
\le
\frac{\mu s r}{(n-s)\ln n}. 
\end{equation}

Combining  (\ref{2023-12-05+12}) with  (\ref{2024-01-04+20}),
(\ref{2024-01-04+21}), (\ref{2023-12-05+10}), (\ref{2024-01-04+25})  we obtain  (\ref{2023-12-05+11}).

%}
 \end{proof}

Now we are ready to prove Lemma \ref{Lemma_1}.

 \begin{proof}[Proof of Lemma \ref{Lemma_1}]
%\end{proof}  
We first derive a convenient asymptotic formula for $m/n$.
Using  
$|\lambda_{n,m,k}|=o(\ln\ln n)$ and  $m=O(n\ln n)$ we obtain by iterating   (\ref{lambda*}) that 
 \begin{align}
 \nonumber
 \frac{m}{n}
 &
 =
\frac{1}{\kappa^*}
\left(
\ln n
+
(k-1)
\ln
\left(
\frac{1}{\kappa^*}
\left(\ln n+(k-1)\ln\frac{m}{n}
-\lambda_{n,m,k}^*
\right)
\right)
-
\lambda_{n,m,k}^*
\right)
\\
\label{2023-12-08+3}
&
=
\frac{1}{\kappa^*}
\left(\ln n+(k-1)\ln\ln n
-(k-1)\ln \kappa^*
-
\lambda_{n,m,k}^*
+ 
O\left(\frac{\ln\ln n}{\ln n}\right)
\right).
\end{align}

Next we observe
 that for any integer $t$  and all sufficiently large $n$ there exists a constant $C_t$ (independent on $r$, $n$ and $m$) such that for each $0\le h\le t$ we have
\begin{align}
\label{2024-01-10}
{\hat q}_{r,s}^{m-h}
\le 
{\hat q}_{r,s}^{m-t}
\le
e^{-r\bigl(\ln n+(k-1)\ln\ln n-\lambda_{n,m,k}^*+ C_t\bigr)},
\qquad
2\le r\le n^{\beta}.
\end{align}
The first inequality is obvious as ${\hat q}_{r,s}\le 1$. Let us show the second inequality. For $t=0$
 and
 $2\le r\le n^{\beta}$ we obtain from
(\ref{2023-12-05+11}) using $1+a\le e^a$ that
\begin{align*}
{\hat q}_{r,s}^{m}
\le
e^{
-m\left(
\frac{r}{n}\kappa^*
-
(2+\mu (s+1))
\frac{r}{(n-s)\ln n}\right)}.
\end{align*}
Then we invoke (\ref{2023-12-08+3}) and write the argument of the exponent in the form 
\begin{equation}
\label{2024-01-10+1}
-r\bigl(\ln n+(k-1)\ln\ln n-\lambda_{n,m,k}^*+ O(1)\bigr).
\end{equation}
In this way we obtain (\ref{2024-01-10}) for $t=0$. Now we show
 (\ref{2024-01-10})  for arbitrary but fixed $t$. We have
 \begin{align*}
{\hat q}_{r,s}^{m-t}
=
{\hat q}_{r,s}^{m(1-tm^{-1})}
\le
e^{-r\bigl(\ln n+(k-1)\ln\ln n-\lambda_{n,m,k}^*+ C_0\bigr)(1-tm^{-1})}.
\end{align*}
A  calculation shows that the argument of the exponent satisfies (\ref{2024-01-10+1}). Hence  (\ref{2024-01-10}) holds.

Let us prove (\ref{component2}). We use the fact (shown in \cite{Daumilas_Mindaugas2023}) that for $\lambda_{n,m,1}\to-\infty$ the probability that  $G_{[n,m]}$ is connected tends to $1$. In view of the  inequality 
$\lambda_{n,m,1}<\lambda_{n,m,k}$ our condition 
$\lambda_{n,m,k}\to-\infty$ implies that $G_{[n,m]}$ is connected with probability $1-o(1)$. In particular, to prove  (\ref{component2}) it suffices to show that $\PP\{{\cal B}'_k\}=o(1)$, where 
${\cal B}'_k={\cal B}_k\cap\{G_{[n,m]}$ is connected$\}$.

Assume that ${\cal B}'_k$ occurs. Then there exists a pair of subsets
$(S,A_r)$ such that  $S\subset {\cal V}$ is of size $s:=|S|\in [k-1]$,  
$A_r\subset {\cal V}\setminus S$ 
is of size $r:=|A_r|\in [2, (n-s)/2]$,
 and $A_r$ induces a connected component in $G_{[n,m]}^{(-S)}$. 
 Moreover, if we choose  a pair with the smallest possible set $S$ 
 then  each $v\in S$ is linked to some vertex $u=u(v)\in A_r$ in $G_{[n,m]}$. 

Let $p_{s,r}$ denote the probability that 
$[s+r]\setminus [s]$ 
induces  a component in $G_{[n,m]}^{-[s]}$
and every $i\in [s]$ is linked to some vertex from $[s+r]\setminus [s]$ 
in $G_{[n,m]}$ (recall that vertex set of $G_{[n,m]}^{-[s]}$ 
is $[n]\setminus [s]$ and  vertex set of $G_{[n,m]}$ is $[n]$).
Let $p^*_{s,r}$ denote the probability that 
$G_{[n,m]}^{-[s]}$ has no edges connecting 
$[s+r]\setminus [s]$  and $[n]\setminus [s+r]$. Note  that 
$p_{s,r}\le p^*_{s,r}$. 

We have, by the union bound and symmetry
\begin{align*}
\PP\{{\cal B}'_k\}
\le 
\sum_{1\le s\le k-1}
\binom{n}{s}
\sum_{2\le r\le (n-s)/2}
\binom{n-s}{r}p_{s,r}
\le S_1+S_2,
\end{align*}
where
\begin{align*}
S_1
=
\sum_{1\le s\le k-1}
\binom{n}{s}
\sum_{2\le r\le n^{\beta}}
\binom{n-s}{r}p_{s,r},
\quad
S_2
=
\sum_{1\le s\le k-1}
\binom{n}{s}
\sum_{n^{\beta} < r\le (n-s)/2}
\binom{n-s}{r}p^*_{s,r}.
\end{align*}
We choose 
$\beta=\max\{1-\frac{\alpha}{2\kappa^*}, \frac{1}{2}\}$. Recall that $\kappa^*\ge \alpha$, see (\ref{2024-01-04}).
To prove the lemma we show that $S_i=o(1)$, for $i=1,2$.

{\it Proof of} $S_1=o(1)$.
Given $s$ and $r$ we evaluate the probability $p_{s,r}$.
Denote $S=[s]$, $U=[s+r]\setminus[s]$. Let $F=(S\cup U,{\cal E}_F)$
be a bipartite graph with the bipartition $S\cup U$ such that each $i\in S$ has degree one. Here ${\cal E}_F$ denotes the edge set of $F$.
Note that $|{\cal E}_F|=s$.

 Fix an integer $1\le h\le s$.
Let 
${\tilde {\cal E}}_{F}=({\cal E}^{(1)},\dots, {\cal E}^{(h)})$ be an ordered partition of the set ${\cal E}_{F}$ 
(every set ${\cal E}^{(i)}$ is nonempty, 
${\cal E}^{(i)}\cap{\cal E}^{(j)}=\emptyset$ 
for $i\not=j$, 
and $\cup_{i=1}^h{\cal E}^{(i)}={\cal E}_{F}$).
Let ${\tilde t}=(t_1,\dots, t_h)\in [m]^h$ be a 
vector with integer valued coordinates satisfying $t_1<\cdots<t_h$. 
Given a pair
$({\tilde {\cal E}}_{F}, {\tilde t})$,
 let ${\cal F}({\tilde {\cal E}}_{F}, {\tilde t})$ denote the event 
 that ${\cal E}^{(i)}\subset {\cal E}_{t_i}$ 
 for each $1\le i\le h$. The event means that the  edges of
$F$ are covered by  the edges of 
$G_{t_1},\dots G_{t_h}$ so 
that
for every $i$ the edge set ${\cal E}^{(i)}$ belongs to the edge set 
${\cal E}_{t_i}$ of $G_{t_i}$
 (we say that ${\cal E}^{(i)}$ receives label $t_i$).
Let $H_{\tilde t}=[m]\setminus\{t_1,\dots, t_h\}$  and let 
${\cal I}(S,U, H_{\tilde t})$ be the event that 
none of the graphs $G_j$, $j\in H_{\tilde t}$ has an edge connecting some 
$v\in U$ and $w\in {\cal V}\setminus (S\cup U)$.

\medskip
Let ${\mathbb F}_s$ denote the set of bipartite graphs with the bipartition $S\cup U$ where each $i\in S$ has degree one.
Note that
 $|{\mathbb F}_s|=r^s$.
We have, by the union bound and independence 
of $G_1,\dots, G_m$, that
\begin{equation}
\label{suma}
p_{s,r}
\le
\sum_{F\in{\mathbb F}_s}
\sum_{({\tilde {\cal E}}_{F},{\tilde t})}
\PP\{{\cal F}({\tilde {\cal E}}_{F}, {\tilde t})\}
\PP\{{\cal I}(S,U, H_{\tilde t})\}.
\end{equation}
Here the second sum runs over all possible pairs 
$({\tilde {\cal E}}_{F}, {\tilde t})$.

Let us estimate the double sum on the right.
We first estimate the probablity  
$\PP\{{\cal I}(S, U, H_{{\tilde t}})\}$ in (\ref{suma}).
By the independence of $G_1,\dots, G_m$, we obtain from (\ref{2024-01-10}) that
\begin{equation}
\label{2023-12-08+4}
\PP\{{\cal I}(S, U, H_{{\tilde t}})\}
=
{\hat q}_{r,s}^{m-h}
\le 
{\hat q}_{r,s}^{m-(k-1)}
\le
e^{-r\bigl(\ln n+(k-1)\ln\ln n-\lambda_{n,m,k}^*+C\bigr)}=:p'_r.
\end{equation}
Here we write $C$ instead of $C_{k-1}$ (see (\ref{2024-01-10})).
Combining (\ref{suma}) and (\ref{2023-12-08+4}) we have
\begin{equation}
\label{suma++++}
p_{s,r}
\le
p'_r
\sum_{F\in{\mathbb F}_s}
S_{F},
\qquad
{\text{where}}
\qquad
S_{F}
:=
\sum_{({\tilde {\cal E}}_{F},{\tilde t})}
\PP\{{\cal F}({\tilde {\cal E}}_{F}, {\tilde t})\}.
\end{equation}
Now we estimate the sum
$S_{F}$. In doing so we use the inequality shown below
\begin{equation}
\label{2024-01-08+2}
\PP\{{\cal F}({\tilde {\cal E}}_{F}, {\tilde t})\}
\le 
\eta^hn^{-h-s}, 
\qquad
{\text{where}}
\qquad
 \eta:=\max\{\eta_j, 2\le j\le k\}.
\end{equation}
For $1\le h\le  s$ and  
a vector $(e_1,\dots, e_h)$ with integer valued coordinates  satisfying $e_1+\cdots+e_h=s$ and $e_i\ge 1$ $\forall i$,
there are 
$\frac{s!}{e_1!\cdots e_h!}$
ordered
partitions 
${\tilde {\cal E}}_{F}=({\cal E}^{(1)},\dots,{\cal E}^{(h)})$ 
of ${\cal E}_{F}$
in $h$ non empty parts of sizes 
$|{\cal E}^{(1)}|=e_1,\dots, |{\cal E}^{(h)}|=e_h$. 
Therefore
we have
 \begin{align*}
 S_{F}
\le
\sum_{h=1}^{s}
\binom{m}{h}
\sum'_{e_1+\dots+e_h=s}
\frac{s!}{e_1!\cdots e_h!}
\frac{\eta^h}{n^{s+h}}
\le
\sum_{h=1}^{s}
\binom{m}{h}
\frac{\eta^h}{n^{s+h}}h^{s}.
\end{align*}
Here $\binom{m}{h}$ counts various labelings 
${\tilde t}=(t_1,\dots, t_h)$.  The sum $\sum'_{e_1+\dots+e_h=s+r-1}$ runs over the set of vectors
$(e_1,\dots, e_h)$ having integer valued coordinates $e_i\ge 1$ 
satisfying $e_ 1+\cdots+e_h=s$.  In the second inequality we used
\begin{displaymath}
\sum'_{e_1+\dots+e_h=s}\frac{s!}{e_1!\cdots e_h!}
\le
 (1+\dots+1)^{s}
 =
 h^{s}.
\end{displaymath}

Invoking inequalities $\binom{m}{h}\le \left(\frac{me}{h}\right)^h$ and 
$\frac{m}{n\ln n}
=\frac{1+o(1)}{\kappa^*}
\le 
\frac{2}{\kappa^*}$ (the last inequality holds for sufficiently large $n$ and $m$, see (\ref{2023-12-08+3})) we obtain
\begin{displaymath}
S_{F}
\le
\frac{1}{n^s}
\sum_{h=1}^{s}
\left(\frac{2 e \eta }{\kappa^*}\right)^h
h^{s-h}\ln^hn
\le
 \left(\frac{2 e \eta s }{\kappa^*}\right)^{s}
 \
 \frac{\ln ^{s}n}{n^s}.
\end{displaymath}
In the last step we used the fact that $\frac{2 e \eta }{\kappa^*}> 1$, see  (\ref{2024-01-04}), and inequality $\sum_{h=1}^sh^{s-h}\le s^s$.
Invoking this bound in (\ref{suma++++})  and using $|{\mathbb F}_s|\le r^s$ we obtain
\begin{equation}
\label{2024-01-10+4}
p_{s,r}
\le 
 p'_r
 \
 r^s
\left(\frac{2 e \eta s}{\kappa^*}\right)^{s}
 \frac{\ln ^{s}n}{n^s}.
 \end{equation}
Finally, using $\binom{n}{s}\le n^s$ and
$\binom{n-s}{r}
\le 
\binom{n}{r}
\le 
\left(\frac{ne}{r}\right)^r$, 
and invoking (\ref{2024-01-10+4})
we have
\begin{align*}
S_1
&
\le \sum_{1\le s\le k-1}
n^s
\sum_{2\le r\le n^{\beta}}
\left(\frac{ne}{r}\right)^r
p_{s,r}
\\
&
\le
\sum_{2\le r\le n^{\beta}}
\left(\frac{ne}{r}\right)^r
p'_r r^{k-1}
\sum_{1\le s\le k-1}
\left(\frac{2 e \eta s}{\kappa^*}\right)^{s}
\ln n^{s}
\\
&
\le 
C''
\sum_{2\le r\le n^{\beta}}
\left(\frac{ne}{r}\right)^r p'_r
r^{k-1}
\ln^{k-1}n.
\end{align*}
Here we upper bounded  $\sum_{1\le s\le k-1}
\left(\frac{2 e \eta s}{\kappa^*}\right)^{s}
\ln n^{s}
\le C'' \ln^{k-1}n$,
where $C''$ does not depend on $r,n,m$. 
Invoking expression (\ref{2023-12-08+4}) of $p'_r$ we obtain
\begin{displaymath}
S_1
\le
C''
\sum_{2\le r\le n^{\beta}}e^{r(\lambda_{n,m,k}^*+O(1))}
=
o(1)
\end{displaymath} 
since $\lambda_{n,m,k}^*\to-\infty$ as $n\to+\infty$.

{\it Proof of (\ref{2024-01-08+2}).}
Given a bipartite graph $F=(S\cup U,{\cal E}_F)$ and  (ordered) partition
${\tilde {\cal E}}_{F}=({\cal E}^{(1)},\dots,{\cal E}^{(h)})$, let
$V^{(i)}$ be the set of vertices incident to the edges from 
${\cal E}^{(i)}$. We denote $e_i=|{\cal E}^{(i)}|$ and
$v_i=|V^{(i)}|$.
For any labeling 
${\tilde t}=(t_{1},\dots, t_{h})$ that assigns labels 
$t_1,\dots, t_h$ to the sets
${\cal E}^{(1)},\dots,{\cal E}^{(r)}$ we have, by the independence of $G_1,\dots, G_m$,
\begin{displaymath}
\PP\{{\cal F}({\tilde {\cal E}}_{F}, {\tilde t})\}
=
\prod_{i=1}^h
\E
\left(
\frac{({\tilde X}_{t_i})_{v_i}}{(n)_{v_i}}Q_{t_i}^{e_i}
\right).
\end{displaymath}
We note that the fraction $\frac{({\tilde X}_{t_i})_{v_i}}{(n)_{v_i}}$ is 
a decreasing function of $v_i$ and it is maximized by 
$\frac{({\tilde X}_{t_i})_{e_i+1}}{(n)_{e_i+1}}$ since
 we always have $v_i\ge e_i+1$. Indeed,  
 given $|{\cal E}^{(i)}|=e_i$ the smallest possible set of vertices
$V^{(i)}$ 
corresponds to the configuration of edges of 
${\cal E}^{(i)}$ that creates a tree subgraph of $F$. Hence $v_i\ge e_i+1$. It follows that 
\begin{equation}
\label{Stirling}
\PP\{{\cal F}({\tilde {\cal E}}_{F}, {\tilde t})\}
\le
\prod_{i=1}^h
\E
\left(
\frac{({\tilde X}_{t_i})_{e_i+1}}{(n)_{e_i+1}}Q_{t_i}^{e_i}
\right).
\end{equation}
We evaluate the product in (\ref{Stirling}).
 Since $e_i\le s$ for each $1\le i\le h$, we upper bound each factor on the right of (\ref{Stirling}) as follows
\begin{equation}
\label{2024-01-08}
\E
\left(
\frac{({\tilde X}_{t_i})_{e_i+1}}{(n)_{e_i+1}}Q_{t_i}^{e_i}
\right)
\le 
\E
\left(
\frac{{\tilde X}_{t_i}^{e_i+1}}{n^{e_i+1}}Q_{t_i}^{e_i}
\right)
\le
\frac{\eta_{e_i+1}}{n^{e_i+1}}
\le 
\frac{\eta}{n^{e_i+1}}.
\end{equation}
Hence the product in  (\ref{Stirling}) is upper bounded by
$\eta^hn^{-h-s}$ (recall that $e_1+\cdots+e_h=s$). We obtain
(\ref{2024-01-08+2}).

{\it Proof of} $S_2=o(1)$.
Given $s$ and $r$ we evaluate the probability $p^*_{s,r}$
for $r$ satisfying $n^{\beta}\le r\le (n-s)/2$.
 We upper bound $q_{r,s}({\tilde X},Q)$ in the expectation below using
 (\ref{sr_1}),
\begin{align*} 
 {\hat q}_{r,s}
 &
 =
 \PP\{X\le 1\}
 +
 \E\left( q_{r,s}({\tilde X},Q) {\mathbb I}_{\{{\tilde X}\ge 2\}}\right)
 \\
 &
 \le 
 1-2\alpha r\frac{n-s-r}{(n-s)(n-s-1)}
 +
 \frac{(s+1)_2}{n}
  \\
  &
   \le 
  1-\frac{\alpha r}{n-s-1}+\frac{(s+1)_2}{n},
\end{align*}
In the last step we used $2\frac{n-s-r}{n-s}\ge 1$ for $r\le (n-s)/2$.
Furthermore, using $1+a\le e^a$ we estimate
\begin{equation}
\label{2024-01-11}
p^*_{s,r}
=
{\hat q}^m_{r,s}
\le
e^{
-m
\left(
\frac{\alpha r}{n-s-1}-\frac{(s+1)_2}{n}
\right)
}
\le
e^{-\alpha r\frac{m}{n}+(s+1)_2\frac{m}{n}}.
\end{equation}

Next,  we  upper bound the binomial coefficient for 
$n^{\beta}\le r\le n/2$
\begin{equation}
\label{2024-01-11+1}
\binom{n}{r}
\le 
\frac{n^n}{r^r(n-r)^{n-r}}
=
e^{n\ln \left(\frac{n}{n-r}\right)
+
r\ln \left(\frac{n-r}{r}\right)}
\le 
e^{2r+(1-\beta)r\ln n}.
\end{equation}
The last inequality is shown in the 
proof of Proposition 1 of \cite{Daumilas_Mindaugas2023}. We omit its proof  here.

Finally, combining (\ref{2024-01-11}), (\ref{2024-01-11+1}) and using 
$\binom{n}{s}\le n^s=e^{s\ln n}$ we obtain
\begin{align*}
\binom{n}{s}
\binom{n}{r}
p^*_{s,r}
&
\le
e^{2r+(1-\beta)r\ln n-\alpha r\frac{m}{n}+(s+1)_2\frac{m}{n}+s\ln n}
\\
&
=
e^{-r\left(\frac{\alpha}{\kappa^*}-(1-\beta)+O(r^{-1})\right)\ln n+O(r)}. 
\end{align*}
In the last step we used (\ref{2023-12-08+3}). 
We note that the  contribution of $s\ln n+(s+1)_2\frac{m}{n}$ is 
accounted in the term $O(r^{-1})$ in the brackets.
Hence 
\begin{displaymath}
S_2
\le 
\sum_{1\le s\le k-1}\sum_{n^{\beta}\le r\le (n-s)/2}
e^{-r\left(\frac{\alpha}{\kappa^*}-(1-\beta)+O(r^{-1})\right)\ln n+O(r)}.
\end{displaymath}
Our choice of $\beta\ge 1-\frac{\alpha}{2\kappa^*}$ yields $\frac{\alpha}{\kappa^*}-(1-\beta)\ge \frac{\alpha}{2\kappa^*}$.
Hence $S_2=o(1)$.
\end{proof}

\subsection{Proof of Lemma \ref{degree}}

We begin with an outline of the proof. We say that $u,v\in {\cal V}$ are linked by community 
$G_i=({\cal V}_i,{\cal E}_i)$ if $u,v\in {\cal V}_i$ and 
${\cal E}_i$ contains the  edge connecting $u$ and $v$ (denoted $\{u,v\}\in {\cal E}_i$).
 The idea of the proof is based on the observation that (in the range of 
$m,n$  considered) given a vertex of degree $k-1$  it is  likely that all of its neighbours are linked to this vertex by different communities.
Motivated by this observation, for $v\in {\cal V}$, we  introduce events
\begin{displaymath}
 {\cal D}(v)=\{S(v)=k-1, \, d_i(v)\in\{0,1\}\,  \forall i\in[m] \},
 \quad
 {\rm{where}}
 \quad
 S(v)=\sum_{i\in [m]}d_i(v).
\end{displaymath}
We say that the vertex $v$ has property ${\cal D}$ if the event ${\cal D}(v)$ occurs.
Note that a vertex with  property ${\cal D}$ has degree (in $G$) at most  $k-1$. Let 
\begin{displaymath}
N'_{k-1}=\sum_{v\in{\cal V}}{\bf I}_{{\cal D}(v)}
\end{displaymath}
denote the number of vertices  having 
property ${\cal D}$. 

We prove  Lemma \ref{degree} in two steps. We firstly 
approximate $N_{k-1}=(1+o_P(1))N'_{k-1}$ as $n\to+\infty$ and show that 
${\bf{E}} N'_{k-1}\to+\infty$ for $\lambda_{n,m,k}\to+\infty$, see  Lemma \ref{approx} below. Then we establish the  concentration  $N'_{k-1}=(1+o_P(1)){\bf{E}} N'_{k-1}$, see
Lemma \ref{concentration}. The proof of Lemma \ref{degree}  is given at the very end of the section.

In the proof below we use the following simple observations. We 
have
\begin{eqnarray}
\nonumber
{\bf{P}}\{d_1(1)=0\}
&
=
&
{\bf{P}}\{1\notin {\cal V}_1\}
+
{\bf{P}}\{1\in {\cal V}_1, d_1(v_1)=0\}
\\
\nonumber
&
=
&
1-{\bf{E}}\frac{{\tilde X}_1}{n}
+
{\bf{E}} \frac{{\tilde X}_1}{n}(1-Q_1)^{{\tilde X}_1-1}
=
1-\frac{\kappa}{n}.
\end{eqnarray}
We similarly show that 
\begin{displaymath}
{\bf{P}}\{d_1(1)=1\}
=
\frac{\tau}{n}.
\end{displaymath}
We observe that these identities imply $\kappa\ge \tau$.

\begin{lem}\label{approx} Let $k\ge 2$ be an integer.
Let $n\to+\infty$. Assume that $m=\Theta(n\ln n)$. Assume that 
 $\tau^*>0$ and
$\eta_2<\infty$. For $k\ge3$ we assume, in addition, that 
$\eta_3<\infty$.
Then  
\begin{eqnarray}
\label{2024-02-01}
&
{\bf{E}} N'_{k-1}
=
(1+o(1))
\frac{(\tau^*)^{k-1}}{(k-1)!}
e^{\lambda_{n,m,k}},
\\
\label{2024-02-02}
&
{\bf{E}}|N_{k-1}-N'_{k-1}| 
=
o\left(e^{\lambda_{n,m,k}}\right)
+o(1). 
\end{eqnarray}
\end{lem} 

{\it Proof of Lemma \ref{approx}.}
Note that random variables $d(v)$,
 $v\in {\cal V}$ are identically distributed and the 
 probabilities ${\bf{P}}\{{\cal D}(v)\}$, 
 $v\in {\cal V}$, are all equal.
 Hence
\begin{equation}
\label{2024-02-03}
{\bf{E}} N'_{k-1}
%=n{\bf{E}}{\bf I}_{{\cal D}(1)}
 = n{\bf{P}}\{{\cal D}(1)\}
 \quad
 {\rm{and}}
 \quad
{\bf{E}}|N_{k-1}-N'_{k-1}| 
\le 
n{\bf{E}}|{\bf I}_{\{d(1)={k-1}\}}-{\bf I}_{{\cal D}(1)}|.
\end{equation}
We obtain (\ref{2024-02-01}) 
from the first identity of (\ref{2024-02-03})
and 
the asymptotic formula 
\begin{equation}
\label{2024-02-02+1}
{\bf{P}}\{{\cal D}(1)\}
=
\bigl(1+o(1)\bigr)
\left(\frac{m}{n}\right)^{k-1}
\frac{(\tau^*)^{k-1}}{(k-1)!}
e^{-\frac{m}{n}\kappa}
=
(1+o(1))\frac{1}{n}\frac{(\tau^*)^{k-1}}{(k-1)!}e^{\lambda_{n,m,k}}.
\end{equation}
Similarly, we derive  (\ref{2024-02-02}) 
from the second inequality  of (\ref{2024-02-03})
and 
 inequalities 
\begin{eqnarray}
\label{2024-02-02+9}
&&
{\bf{E}}\left|{\bf I}_{\{d(1)=k-1\}}
-
{\bf I}_{\{S(1)=k-1\}}\right|
=
O\left(\frac{m^4}{n^6}\right)
+
O\left(\frac{m^{k}}{n^{k+1}}e^{-\frac{m}{n}\kappa}\right),
\\
\label{2024-02-02+3}
&&
{\bf{E}}\left|{\bf I}_{\{S(1)=k-1\}}
-
{\bf I}_{{\cal D}(1)}\right|
=
O\left(\left(\frac{m}{n}\right)^{k-2}e^{-\frac{m}{n}\kappa}\right).
\end{eqnarray}
We note that since $\tau^*>0$ bound  (\ref{2024-02-02+3}) combined with  the first relation of
 (\ref{2024-02-02+1})  implies
\begin{equation} 
\label{2024-02-03+10}
{\bf{P}}\{S(1)=k-1\}
=
(1+o(1))\left(\frac{m}{n}\right)^{k-1} 
\frac{(\tau^*)^{k-1}}{(k-1)!}
e^{-\frac{m}{n}\kappa}.
\end{equation}
In the remaining part of the proof we show  (\ref{2024-02-02+1}), (\ref{2024-02-02+9}),
(\ref{2024-02-02+3}). Before the proof we introduce some notation. 
For $K\subset [m]$ we denote $S_K=\sum_{i\in K}d_i(1)$ and 
introduce events
\begin{displaymath}
{\cal A}_K=\{d_i(1)\ge 1, \ \forall i\in K\},
\qquad
{\cal L}_K=
 \{S_K=k-1\}
\cap 
{\cal A}_K
\cap 
 \{S_{[m]\setminus K}=0\}.
\end{displaymath}

Proof of (\ref{2024-02-02+1}). Event 
${\cal D}(1)$ is the union of mutually disjoint events
\begin{displaymath}
 {\cal D}(1)
 =
\bigcup_{K:\, |K|=k-1}
{\cal L}_K.
\end{displaymath}
We have, by symmetry and independence and identical distribution   of  
$d_1(1)$, $\dots$, $d_m(1)$, that
 \begin{eqnarray}
\nonumber
{\bf{P}}\{{\cal D}(1)\}
&
=
&
{m\choose {k-1}}
%\binom{m}{k-1}
{\bf{P}}\{S_{\{1,\dots, k-1\}}=k-1, {\cal A}_{\{1,\dots, k-1\}}\}
{\bf{P}}\{ S_{\{k,\dots,m\}}=0\}
\\
\nonumber
&
=
&
{m\choose{k-1}}
%\binom{m}{k-1}
\left({\bf{P}}\{d_1(v_1)=1\}\right)^{k-1}
\left({\bf{P}}\{d_1(v_1)=0\}\right)^{m-k+1}
\\
\label{2024-02+02+2}
&
=
&
{m\choose{k-1}}
%\binom{m}{k-1}
\left(\frac{\tau}{n}\right)^{k-1}
\left(1-\frac{\kappa}{n}\right)^{m-k+1}.
\end{eqnarray}
Invoking   in (\ref{2024-02+02+2}) approximations 
${m\choose{k-1}}
%\binom{m}{k-1}
=\frac{m^{k-1}}{(k-1)!}(1+o(1))$, 
$\tau=\tau^*+o(1)$, 
see Fact 2, 
and the identity, where we use $\ln (1-x)=-x+O(x^2)$ for $x=o(1)$,
\begin{equation}
\label{2024-02-02+6}
\left(1-\frac{\kappa}{n}\right)^{m-k+1}
=
e^{(m-k+1)\ln\left(1-\frac{\kappa}{n}\right)}
=
e^{-\frac{m}{n}\kappa+O\left(\frac{1}{n}+\frac{m}{n^2}\right)}
\end{equation}
 we obtain the first relation of (\ref{2024-02-02+1}).
The second one follows from the  identity 
$\frac{m^{k-1}}{n^{k-2}}e^{-\frac{m}{n}\kappa}=e^{-\lambda_{n,m,k}}$.
 
% \end{document}
 
 Proof of (\ref{2024-02-02+3}).
 We represent 
$\{S(1)=k-1\}$ by unions of mutually disjoint events
\begin{displaymath}
 \{S(1)=k-1\}
 =
\bigcup_{K\subset[m]:\, 1\le |K|=k-1} 
{\cal L}_K
={\cal D}(1)\cup{\cal R}(1),
\end{displaymath}
where 

\begin{displaymath}
{\cal R}(1)
=
\bigcup_{h=1}^{k-2}\bigcup_{K\subset[m], |K|=h}{\cal L}_K.
\end{displaymath}
Hence 
\begin{equation}
\label{2024-02-02+7}
0
\le 
{\bf I}_{\{S(1)=k-1\}}-{\bf I}_{{\cal D}(1)}
\le 
{\bf I}_{{\cal R}(1)}.
\end{equation} 
Using symmetry and  the independence of $d_1(1),\dots, d_m(1)$ we evaluate the expectation
\begin{eqnarray}
\nonumber
{\bf{E}} {\bf I}_{{\cal R}(1)}
&
=
&
{\bf{P}}\{{\cal R}(1)\}
=
\sum_{h=1}^{k-2}
{m\choose h}
%\binom{m}{h}
{\bf{P}}\left\{S_{[h]}=k-1,{\cal A}_{[h]}\right\}
{\bf{P}}\left\{S_{[m]\setminus[h]}=0\right\}
\\
\label{2024-02-02+4}
&
\le
&
\sum_{h=1}^{k-2}
%\binom{m}{h}
{m\choose h}
\left(\frac{\kappa}{n}\right)^h
\left(1-\frac{\kappa}{n}\right)^{m-h}.
\end{eqnarray}
Here $[h]$ stands for the set $\{1,\dots, h\}$. In the last inequality we invoked  identity
\begin{displaymath}
{\bf{P}}\left\{S_{[m]\setminus[h]}=0\right\}
=
\prod_{i\in[m]\setminus[h]}{\bf{P}}\{d_i(1)=0\}
%=\left({\bf{P}}\{d_1(1)=0\}\right)^{m-h}
=\left(1-\frac{\kappa}{n}\right)^{m-h}
\end{displaymath}
and used  inequalities
\begin{displaymath}
{\bf{P}}\left\{S_{[h]}=k-1,{\cal A}_{[h]}\right\}
\le {\bf{P}}\{{\cal A}_{[h]}\}
=\prod_{i\in [h]}{\bf{P}}\{d_i(1)\ge 1\}
%=\left({\bf{P}}\{d_1(1)\ge 1\}\right)^h
=\left(\frac{\kappa}{n}\right)^h.
\end{displaymath}
Using 
$\left(1-\frac{\kappa}{n}\right)^{m-h}
\le 
e^{-\frac{\kappa}{n}(m-h)}
=
(1+o(1))e^{-\frac{m}{n}\kappa}$ and the fact that $n=o(m)$ 
we upper bound the quantity in (\ref{2024-02-02+4})
 by
$O\left(\left(\frac{m}{n}\right)^{k-2}
e^{-\frac{m}{n}\kappa}\right)$. 
Now (\ref{2024-02-02+3}) follows from (\ref{2024-02-02+7}),
(\ref{2024-02-02+4}).

Proof of  (\ref{2024-02-02+9}). 
Let ${\cal B}$ denote the event that vertex $1$ is adjacent to some $u\in{\cal V}$ in two communities simultaneously, 
\begin{displaymath}
{\cal B}=\bigl\{
\{1,u\}\in {\cal E}_i\cap {\cal E}_j
\quad
{\rm {for some}}
\quad
 u\in {\cal V}\setminus \{1\}
 \quad
 {\rm{ and some}}
 \quad
  i<j\bigr\}.
  \end{displaymath}
We observe that on the complement event ${\bar {\cal B}}$
we have $d(1)=S(1)$. Hence
\begin{displaymath}
d(1)
=
\bigl(
{\bf I}_{\cal B}
+
{\bf I}_{\bar {\cal B}}
\bigr)
d(1)
=
{\bf I}_{\cal B} d(1)
+
{\bf I}_{\bar {\cal B}} S(1)
=
S(1)-R_1,
\quad
R_1:={\bf I}_{\cal B} (S(1)-d(1)).
\end{displaymath}
Furhermore, since $R_1=0$ implies $d(1)=S(1)$ we have
\begin{displaymath}
\left|
{\bf I}_{\{d(1)=k-1\}}
-
{\bf I}_{\{S(1)=k-1\}}
\right|
\le 
{\bf I}_{\{d(1)=k-1\}}{\bf I}_{\{R_1\ge 1\}}
+
{\bf I}_{\{S(1)=k-1\}}{\bf I}_{\{R_1\ge 1\}}.
\end{displaymath}
Taking the expected values of both sides we obtain
\begin{eqnarray}
\nonumber
&&
{\bf{E}} \left|
{\bf I}_{\{d(1)=k-1\}}
-
{\bf I}_{\{S(1)=k-1\}}
\right|
\\
\nonumber
&&
\le 
{\bf{P}}\{d(1)=k-1,\, R_1\ge 1\}
+
{\bf{P}}\{S(1)=k-1,\, R_1\ge 1\}
=:
p_1+p_2.
\end{eqnarray}
To prove (\ref{2024-02-02+9}) we show that 
\begin{equation}
\label{2023-07-21+5}
p_i
=
O\left(\frac{m^4}{n^6}\right)
+
O\left(\frac{m^{k}}{n^{k+1}}e^{-\frac{m}{n}\kappa}\right),
\qquad
i=1,2.
\end{equation}
We only  prove (\ref{2023-07-21+5}) for $i=1$. For $i=2$ the proof 
is much the same.

Let $N(1)$ denote the set of neighbours of vertex $1$ in $G$. For $d(1)=k-1$ we have $|N(1)|=k-1$. Let 
$N^*(1)=(u_1^*,\dots, u_{k-1}^*)$ be a random permutation of elements of 
$N(1)$. Then $N(1)=\{u_1^*,\cdots, u_{k-1}^*\}$.
Let 
$\gamma_{r}=\sum_{i\in[m]}{\bf I}_{\{\{1,u_r^*\}\in{\cal E}_i\}}$ be the number of communities $G_i$ where $1$ and $u_r^*$ are adjacent.
For $r=1,2,\dots, k-1$
introduce events   
\begin{displaymath}
{\cal H}_{r}=\{\gamma_{r}\le 2\},
\ \
{\cal H}_{r,2}=\{\gamma_r=2\},
\ \
{\cal H}_{r*}= {\cal H}_{r,2}\cap
 \bigl\{
\gamma_j=1, \forall j\in[k-1]\setminus\{r\}\bigr\}.
\end{displaymath}
Using the fact that  events $\bigcap_{r=1}^{k-1} {\cal H}_r$
 and 
$\bigcup_{r=1}^{k-1} {\bar {\cal H}}_r$ are complement to each other, 
we write
\begin{eqnarray}
\nonumber
p_1
&
 =
&
 {\bf{P}}\left\{d(1)=k-1, R_1\ge 1, \bigcap_{r=1}^{k-1} {\cal H}_r\right\}
 +
 {\bf{P}}\left\{d(1)=k-1, R_1\ge 1, \bigcup_{r=1}^{k-1} {\bar {\cal H}}_r\right\}
\\
\nonumber
&
=:
& I_1+I_2.
  \end{eqnarray}
Now assume that event 
$\{R_1\ge 1\}\cap\left( \bigcap_{r=1}^{k-1} {\cal H}_r\right)$ occurs. Then  either there is a 
single $\gamma_r$ attaining value $2$ (while remaining $\gamma_j$, with $j\not=r$, attain 
value $1$) or there are (at least) two $\gamma$'s, say $\gamma_{s}$ and 
$\gamma_{t}$, attaining value $2$. Note that the second alternative only makes sense for $k\ge 3$.  Consequently,
\begin{eqnarray}
\nonumber
I_1
&
\le
&
{\bf{P}}
\left\{
d(1)=k-1,\bigcup_{r=1}^{k-1} {\cal H}_{r*}
\right\}
+
  {\bf{P}}
\left\{
d(1)=k-1,  \bigcup_{\{s,t\}\subset [k-1]}
   {\cal H}_{s,2}\cap {\cal H}_{t,2}
   \right\}
\\
\nonumber
&
=:
&
I_3+I_4,
\end{eqnarray}
where $I_4=0$ for $k=2$.
Let us upper bound $I_2,I_3,I_4$.
We have, by the union bound and symmetry,
\begin{eqnarray}
\nonumber
I_2
&
\le 
&
\sum_{r=1}^{k-1}
 {\bf{P}}
 \left\{
 d(1)=k-1, {\bar {\cal H}}_r\right\}
 =
 (k-1)I_2',
 \qquad
 I_2':={\bf{P}}\left\{d(1)=k-1, \gamma_1\ge 3\right\},
 \\
 \nonumber
I_3
&
=
&
\sum_{r=1}^{k-1}
{\bf{P}}
\left\{
d(1)=k-1, {\cal H}_{r*}
\right\}
=
(k-1)
I_3',
\qquad
I_3':=
{\bf{P}}
\left\{
d(1)=k-1, {\cal H}_{1*}
\right\},
\\
\nonumber
I_4
&
\le
&
\sum_{\{s,t\}\subset [k-1]}
 {\bf{P}}
\left\{
d(1)=k-1,
   {\cal H}_{s,2}\cap {\cal H}_{t,2}
   \right\}
   =
 {{k-1}\choose2}
 %  \binom{k-1}{2}
   I_4',
   \\
   \nonumber
   I_4'
   &
   :=
   &
   {\bf{P}}
\left\{
d(1)=k-1,
   {\cal H}_{1,2}\cap {\cal H}_{2,2}
   \right\}.
   \end{eqnarray}
To show (\ref{2023-07-21+5}) we upper bound  probabilities $I_2'$, $I_3'$ and $I_4'$.
The bound  (\ref{2023-07-21+5}) follows from  respective bounds 
(\ref{2024-02-05+5}), (\ref{2024-02-05+6}) and (\ref{2024-02-05+7}) shown below.

Let us estimate $I_2'$.
 The event $\{d(1)=k-1, \gamma_1\ge 3\}$ implies that for some 
 $u\in{\cal V}\setminus\{1\}$ 
 and some
 $\{j_1,j_2,j_3\}\subset [m]$ 
 we have
 $\{1,u\}\in{\cal E}_{j_1}\cap {\cal E}_{j_2}\cap {\cal E}_{j_3}$.
 We have, by the union bound,
 \begin{eqnarray}
 \label{2023-07-17+6}
 I'_2
&
 \le
 & (n-1)
 {m\choose 3}
 %\binom{m}{3}
 {\bf{P}}\bigl\{
 \{1,u\}\in{\cal E}_{j_1}\cap {\cal E}_{j_2}\cap {\cal E}_{j_3}
 \bigr\}
 \\
 \label{2024-02-05+5}
 &
 =
 &
 (n-1)
 {m\choose 3}
 %\binom{m}{3}
 \left({\bf{E}} \left(\frac{({\tilde X})_2}{(n)_2}Q\right)\right)^3
 %=
 %O\left(\frac{m^3}{n^5}{\tilde \eta}_2^3\right)
 =
 O\left(\frac{m^3}{n^5}\right)
 =
 o\left(\frac{m^4}{n^6}\right).
 \end{eqnarray}
% \end{document}
Let us estimate $I_3'$. 
  Recall that $d_1(1)$ and $d_2(1)$ denote the degrees of vertex \
  $1$ in $G_1$ and $G_2$ respectively. 
  Given $v\in{\cal V}\setminus\{1\}$, integers $s,t\ge 0$, 
  and $\{i,j\}\subset [m]$, 
 introduce events  ${\cal C}=\{ \gamma_r=1,\ 2\le r\le k-1\}$,
\begin{eqnarray}
\nonumber
{\cal B}_{i,j}(v)
&
=
&
\Bigl\{u_1^*=v,
\,
\{1,v\}\in{\cal E}_{i}\cap {\cal E}_{j},
\
\{1,v\}\notin{\cal E}_{h} \ \forall h\in [m]\setminus\{i,j\}\Bigr\}
\cap
{\cal C},
\\
\nonumber
{\cal B}^*_{s,t}(v)
 &
 =
 & 
 \{d(1)=k-1, {\cal B}_{1,2}(v), d_1(1)=1+s, d_2(1)=1+t\}
 \cap
{\cal C}.
\end{eqnarray}
Fix $u\in {\cal V}\setminus\{1\}$. We have, by symmetry,
 \begin{eqnarray}
 \label{2023-07-20+2}
  I_3'
  &=&
  \sum_{v\in{\cal V}\setminus\{1\}}\sum_{\{i,j\}\subset [m]}
 {\bf{P}}\{d(1)=k-1, {\cal B}_{i,j}(v)\}
\\
\nonumber
&=&
 (n-1)
 {m\choose 2}
 %\binom{m}{2}
  {\bf{P}}\{d(1)=k-1, {\cal B}_{1,2}(u)\}.
  \end{eqnarray}
 % \end{document}
 Let us evaluate  the probability on the right
  \begin{equation}
  \label{2023-07-20+1}
  {\bf{P}}\{d(1)=k-1, {\cal B}_{1,2}(u)\}
  =\sum_{(s,t):\,0\le s+t\le k-2}
  {\bf{P}}\{
  {\cal B}^*_{s,t}(u)\}.
  \end{equation}
  Consider the graph $G^{-\{1,2\}}$ with vertex set ${\cal V}$ and  edge set $\cup_{j=3}^m{\cal E}_j$.  Assume that  event
$
 {\cal B}^*_{s,t}(u)
 $ occurs. Then 
  the degree of vertex $1$ in $G^{-\{1,2\}}$  (denoted 
  $d^{-\{1,2\}}(1)$) equals $k-2-s-t$. Moreover, we have $d^{-\{1,2\}}(1)=\sum_{j=3}^md_j(1)$ (since $\gamma_{r}=1$ for $2\le r\le k-1$). Hence
  \begin{eqnarray}
  \nonumber
 {\bf{P}} \{
  {\cal B}^*_{s,t}(u)
 \}
&
 \le
 & 
 {\bf{P}}\left\{
 \{1,u\}\in {\cal E}_1\cap {\cal E}_2,
\sum_{j=3}^md_j(1)=k-2-s-t\right\}
 \\
 \label{2024-02-05}
 &
 =
 &
 {\bf{P}}\bigl\{
 \{1,u\}\in {\cal E}_1\cap {\cal E}_2\bigr\}
 {\bf{P}}\left\{\sum_{j=3}^md_j(1)=k-2-s-t\right\}.
 \end{eqnarray}
 Here we used the independence of the random sets 
 ${\cal E}_1,\dots, {\cal E}_{m}$.
 The first probability of (\ref{2024-02-05})
 \begin{equation}
  \label{2024-02-05+1}
  {\bf{P}}
  \bigl\{
 \{1,u\}\in {\cal E}_1\cap {\cal E}_2
 \bigr\}
 =
 \left(
 {\bf{E}}\left(\frac{({\tilde X})_2}{(n)_2}Q\right)\right)^2
 =
 O(n^{-4}).
 \end{equation}
 For $k-2-s-t\ge 1$ the second probability of (\ref{2024-02-05}) is evaluated  in the same way as the probability
 ${\bf{P}}\{S(1)=k-1\}$ in (\ref{2024-02-03+10}). Now we have 
 $S'(1):=\sum_{j=3}^md_j(1)$ instead of $S(1)=\sum_{j=1}^md_j(1)$ and we have $h=k-2-s-t$ instead of $k-1$.
 For $h=1,\dots, k-2$ the argument of the proof of (\ref{2024-02-03+10}) applies
 to
 ${\bf{P}}\{S'(1)=h\}$ and  we have 
 \begin{displaymath}
  {\bf{P}}\left\{S'(1)=h\right\}
  =
  \left(1+o(1)\right)
\left(\frac{m}{n}\right)^{h}
\frac{(\tau^*)^{h}}{h!}
e^{-\frac{m}{n}\kappa}.
\end{displaymath}
Furthermore, for $h=0$ we have 
\begin{displaymath}
{\bf{P}}\left\{S'(1)=0\right\}
=\left({\bf{P}}\{d_3(1)=0\}\right)^{m-2}
=
\left(1-\frac{\kappa}{n}\right)^{m-2}
=
(1+o(1))e^{-\frac{m}{n}\kappa}.
\end{displaymath}
Next, using the fact that $\max_{0\le h\le k-2}\frac{m^h}{n^h}=\frac{m^{k-2}}{n^{k-2}}$ we obtain the bound 
\begin{displaymath}
 {\bf{P}}\left\{S'(1)=h\right\}
  =
 O\left(
\left(\frac{m}{n}\right)^{k-2}
e^{-\frac{m}{n}\kappa}
\right),
\qquad 
h=0,1,\dots, k-2.
\end{displaymath}
Hence 
 the second probability of (\ref{2024-02-05}) is 
 $ O\left(
\left(\frac{m}{n}\right)^{k-2}
e^{-\frac{m}{n}\kappa}
\right)$.
Combining 
 this bound with  (\ref{2024-02-05+1})  
 we obtain
$
 {\bf{P}} \{
  {\cal B}^*_{s,t}(u)
 \}
 =
O\left(\frac{m^{k-2}}{n^{k+2}}
e^{-\frac{m}{n}\kappa}
\right)$.
Now 
 (\ref{2023-07-20+1}) implies the bound  
\begin{displaymath}
{\bf{P}}\{d(1)=k-1, {\cal B}_{1,2}(u)\}=O\left(\frac{m^{k-2}}{n^{k+2}}
e^{-\frac{m}{n}\kappa}
\right).
\end{displaymath}
Finally, 
(\ref{2023-07-20+2}) implies the bound
\begin{equation}
\label{2024-02-05+6}
I_3'
= 
O\left(
 \frac{m^{k}}{n^{k+1}}e^{-\frac{m}{n}\kappa}
 \right).
 \end{equation}

Let us estimate $I_4'$. 
Assume that  event  
$\{d(1)=k-1\}\cap {\cal H}_{1,2}\cap {\cal H}_{2,2}$ 
occurs. Then for some $\{u,v\}\subset {\cal V}\setminus\{1\}$ one of the following alternatives holds:

${\cal A}_1$: for some $i_1\not=i_2$  we have 
$\{1,u\},\{1,v\}\in {\cal E}_{i_1}\cap {\cal E}_{i_2}$;

${\cal A}_2$: for some $i_1\not=i_2\not=i_3$ we have 
$\{1,u\},\{1,v\}\in{\cal E}_{i_1}$ and $\{1,u\}\in{\cal E}_{i_2}$,
 and
$\{1,v\}\in{\cal E}_{i_3}$;

${\cal A}_3$: for some $i_1\not=i_2\not=i_3\not=i_4$ we have 
$\{1,u\}\in{\cal E}_{i_1}\cap{\cal E}_{i_2}$ 
and $\{1,v\}\in{\cal E}_{i_3}\cap{\cal E}_{i_4}$.

\noindent
Given $\{u,v\}$, we estimate the probabilities ${\bf{P}}\{{\cal A}_i\}$, 
$1\le i\le 3$, using  the union bound and symmetry,
\begin{eqnarray}
\nonumber
{\bf{P}}\{{\cal A}_1\}
&\le&
{m\choose 2}
%\binom{m}{2}
{\bf{P}}\{\{1,u\}, \{1,v\}\in {\cal E}_{i_1}\}
{\bf{P}}\{\{1,u\}, \{1,v\}\in {\cal E}_{i_2}\}
\\
\nonumber
&=&
{m\choose 2}
%\binom{m}{2}
\left(
{\bf{E}}
\left(\frac{({\tilde X})_3}{(n)_3}Q^2\right)
\right)^2,
\\
\nonumber
{\bf{P}}\{{\cal A}_2\}
&\le&
(m)_3
\left(
{\bf{E}}
\left(\frac{({\tilde X})_3}{(n)_3}Q^2\right)
\right)
\left(
{\bf{E}}
\left(\frac{({\tilde X})_2}{(n)_2}Q\right)
\right)^2,
\\
\nonumber
{\bf{P}}\{{\cal A}_3\}
&\le&
{m\choose 2}
%\binom{m}{2}
{{m-2}\choose 2}
%\binom{m-2}{2}
\left(
{\bf{E}}
\left(\frac{({\tilde X})_2}{(n)_2}Q\right)
\right)^4.
\end{eqnarray}
Furthermore, taking into account that there are $(n-1)_2$
ways to choose vertices $u\not=v$, we have
\begin{eqnarray}
\nonumber
I_4'
&\le& 
(n-1)_2\bigl(
{\bf{P}}\{{\cal A}_1\}
+
{\bf{P}}\{{\cal A}_2\}
+
{\bf{P}}\{{\cal A}_3\}
\bigr)
\\
\label{2024-02-05+7}
&=&
O
\left(
\frac{m^2}{n^4}+\frac{m^3}{n^5}+\frac{m^4}{n^6}
\right)
=
O\left(\frac{m^4}{n^6}\right).
\end{eqnarray}
The latter bound combined with (\ref{2023-07-17+6}) and  (\ref{2024-02-05+6})  yields  (\ref{2023-07-21+5}). $\qed$

\begin{lem}\label{concentration} Let $k\ge 2$ be an integer.
Let $n\to+\infty$. Assume that $m=\Theta(n\ln n)$. Assume that $\alpha>0$, $\tau^*>0$, $\eta_2<\infty$, 
$\mu'<\infty$.
%{\color{green}For $k\ge 3$ we assume, in addition, that $\eta_3<\infty$ AR TIKRAI REIKIA?}.
Assume that 
$\lambda_{n,m,k}\to+\infty$. 
 Then  
$N'_{k-1}=(1+o_P(1)){\bf{E}} N'_{k-1}$.
\end{lem} 
In the proof of Lemma \ref{concentration} we use the following fact.

{\bf Fact 3}.
%\label{Remark20204-02-08} 
Let $n\to+\infty$.  
For $m=\Theta(n\ln n)$  
condition $\mu'<\infty$ implies 
\begin{equation}
\label{2024-02-08}
{\bf{E}}
\left(
({\tilde X})_2
\left(
1-(1-Q)^{{\tilde X}-1}
\right)
\right)=o(n^2/m).
\end{equation}

{\it Proof of Fact 3.}
 Inequalities $1-(1-q)^{x}\le 1$ and $1-(1-q)^{x}\le qx$ imply inequality 
$(1-(1-q)^x\le \min\{1,qx\}$.
From the latter inequality we obtain
\begin{displaymath}
{\bf{E}}
\left(
({\tilde X})_2
\left(
1-(1-Q)^{{\tilde X}-1}
\right)
\right)
\le 
{\bf{E}} \bigl({\tilde X}^2\min\{1,{\tilde X} Q\}\bigr)=:I.
\end{displaymath}
We will show that $I=o(n/\ln n)$.
We split
\begin{displaymath}
I
=
{\bf{E}}\left({\tilde X}^2\min\{1,{\tilde X} Q\}{\bf I}_{\{X<\sqrt n\}}\right)
+
{\bf{E}}\left({\tilde X}^2\min\{1,{\tilde X} Q\}{\bf I}_{\{X\ge\sqrt n\}}\right)
=:I_1+I_2.
\end{displaymath}
Using $x/\ln (1+x)\le \sqrt n/ \ln(1+\sqrt n)$ for $x<\sqrt n$ and
$x/\ln(1+x)\le n/\ln (1+n)$ for $x\le n$ we upper bound
\begin{eqnarray}
\nonumber
I_1
&
\le 
&
\frac{\sqrt n}{\ln(1+\sqrt n)}
{\bf{E}}
\left(X\min\{1,XQ\}\ln(1+X){\bf I}_{\{X<\sqrt n\}}\right)
\le
\frac{\sqrt n}{\ln(1+\sqrt n)}\mu',
\\
\nonumber
I_2
&
\le
&
\frac{n}{\ln(1+n)}
{\bf{E}}
\left({\tilde X}\min\{1,{\tilde X} Q\}\ln(1+{\tilde X}){\bf I}_{\{X\ge \sqrt n\}}\right)
\le
\frac{n}{\ln(1+n)}I_2',
\end{eqnarray}
where 
\begin{displaymath}
I_2'={\bf{E}}
\left(X\min\{1,X Q\}\ln(1+X){\bf I}_{\{X\ge \sqrt n\}}\right).
\end{displaymath}
Our condition $\mu'<\infty$ implies $I_2'=o(1)$
as $n\to+\infty$.
Hence 
\begin{displaymath}
\qquad
\qquad
I
\le
 \frac{\sqrt n}{\ln(1+\sqrt n)}
 \mu'
 +
 \frac{n}{\ln(1+n)}
 I_2'
 =
 o(n/\ln n).
 \qquad
 \quad 
 \qed
\end{displaymath}

{\it Proof of Lemma \ref{concentration}}.
To show the concentration of $N'_{k-1}$ around the mean value 
${\bf{E}} N'_{k-1}$ we upper bound the variance of $N'_{k-1}$. To this aim we evaluate the covariances 
${\bf{Cov}} (I_{{\cal D}(v)},I_{{\cal D}(u)})$.

Given vertex $v\in {\cal V}$ 
and set
$K\subset [m]$ of size $|K|=k-1$  denote the  event 
\begin{displaymath}
{\cal D}_K(v)=\{d_i(v)=1 \
\forall i\in K
\
{\rm{and}}
\
d_j(v)=0\ \forall j\in [m]\setminus K\}.
\end{displaymath}
Note that for  $K\not= K'$   events 
${\cal D}_K(v), {\cal D}_{K'}(v)$ are mutually disjoint. Hence
\begin{displaymath}
{\bf I}_{{\cal D}(v)}
=
\sum_{K\subset [m], |K|=k-1}{\bf I}_{{\cal D}_K(v)}.
\end{displaymath}
For $h=0,1,\dots, k-1$ we denote $K(h)=\{h+1,\dots, h+k-1\}$.  Observe that   $K(0)$ and $K(h)$ share  
$|K(0)\cap K(h)|=k-1-h$ common elements. We have, by symmetry,
\begin{eqnarray}
\label{2023-07-25+2}
{\bf{E}} \bigl({\bf I}_{{\cal D}(v)}{\bf I}_{{\cal D}(u)}\bigr)
&
=
&
%\binom{m}{k-1}
{m\choose{k-1}}
{\bf{E}} 
\left(
{\bf I}_{{\cal D}_{K(0)}(v)}
\sum_{K\subset [m],\, |K|=k-1}{\bf I}_{{\cal D}_{K}(u)}
\right)
\\
\nonumber
&
=
&
%\binom{m}{k-1}
{m\choose{k-1}}
\sum_{h=0}^{k-1}
%\binom{k-1}{k-1-h}\binom{m-k+1}{h}
{{k-1}\choose{k-1-h}}{{m-k+1}\choose h}
{\bf{E}} \bigl({\bf I}_{{\cal D}_{K(0)}(v)}{\bf I}_{{\cal D}_{K(h)}(u)}\bigr).
\end{eqnarray}
Let us evaluate   
${\bf{E}} 
\bigl(
{\bf I}_{{\cal D}_{K(0)}(v)}
{\bf I}_{{\cal D}_{K(h)}(u)}\bigr)
=
{\bf{P}}\{{\cal D}_{K(0)}(v)\cap {\cal D}_{K(h)}(u)\}$.
To this aim we write event 
${\cal D}_{K(0)}(v)\cap {\cal D}_{K(h)}(u)$ in the form
\begin{displaymath}
{\cal X}_{K(0)\cap K(h)}\cap{\cal Y}_{[m]\setminus(K(0)\cup K(h))}
\cap {\cal Z}_{K(0)\setminus K(h)}\cap{\cal W}_{K(h)\setminus K(0)},
\end{displaymath}
where  for any $A\subset [m]$ we denote  events
\begin{eqnarray}
\nonumber
&&
{\cal X}_A=\{d_i(v)=d_i(u)=1 \ \forall i\in A\},
\qquad
{\cal Y}_A=\{d_i(v)=d_i(u)=0 \ \forall i\in A\},
\\
\nonumber
&&
{\cal Z}_A=\{d_i(v)=1, d_i(u)=0 \ \forall i\in A\},
\qquad
{\cal W}_A=\{d_i(v)=0, d_i(u)=1 \ \forall i\in A\}.
\end{eqnarray}
By the independence and identical distribution of $G_1,\dots, G_m$, we have
\begin{eqnarray}
\nonumber
{\bf{P}}\{{\cal D}_{K(0)}(v)\cap {\cal D}_{K(h)}(u)\}
&&=
{\bf{P}}\{{\cal X}_{K(0)\cap K(h)}\}
\times
{\bf{P}}\{{\cal Y}_{[m]\setminus(K(0)\cup K(h))}\}
\\
\nonumber
&&
\
\
\
\times 
{\bf{P}}\{ {\cal Z}_{K(0)\setminus K(h)}\}
\times
{\bf{P}}\{{\cal W}_{K(h)\setminus K(0)}\}
\\
\label{2023-07-25+1}
&&=
q_1^{k-1-h}q_2^{m-k-h+1}q_3^{2h}.
\end{eqnarray}
Here we denote
\begin{eqnarray}
\nonumber
&&
q_1={\bf{P}}\{d_1(v)=d_1(u)=1\},
\qquad
q_2={\bf{P}}\{d_1(v)=d_1(u)=0\},
\\
\nonumber
&&
q_3={\bf{P}}\{d_1(v)=1, d_1(u)=0\}.
\end{eqnarray}
We show below that
\begin{equation}
\label{2024-02-06}
q_1\le \frac{1}{n}{\bf{E}}\left({\tilde X}_1^2Q_1\right),
\qquad
q_3\le \frac{\tau}{n},
\qquad
q_2=1-2\frac{\kappa}{n}+\frac{\Delta_2}{(n)_2},
\end{equation}
where $\Delta_2$ satisfies  
$0\le \Delta_2\le
{\bf{E}}
\left(
({\tilde X}_1)_2
\left(
1-(1-Q_1)^{{\tilde X}_1-1}
\right)
\right)$.

Using (\ref{2024-02-06}) we  upper bound the product in (\ref{2023-07-25+1}).
Firstly, combining  $1+a\le e^a$, (\ref{2024-02-06}) and (\ref{2024-02-08})
we estimate
\begin{displaymath}
q_2^m
\le 
\left(e^{-2\frac{\kappa}{n}+\frac{\Delta_2}{(n)_2}}\right)^m
=
e^{-2\kappa\frac{m}{n}}\left(1+O\left(\Delta_2\frac{m}{(n)_2}\right)\right)
=
\left(1+o(1)\right)
e^{-2\kappa\frac{m}{n}}.
\end{displaymath}
This bound extends to $q_2^{m-t}$ for small $t$. In particular, for $0\le t\le 2k-2$, we have
\begin{eqnarray}
\nonumber
 q_2^{m-t}
 &
 =
 \left(q_2^m\right)^{1-\frac{t}{m}}
 \le 
  \left(1+o(1)\right)
 e^{-2\kappa\frac{m}{n}
 \left(1-\frac{t}{m}\right)}
\\
\nonumber
&
=(1+o(1))\left(1+O(n^{-1})\right)
e^{-2\kappa\frac{m}{n}}
\\
\label{2024-02-07+1}
&
=
(1+o(1))
e^{-2\kappa\frac{m}{n}}.
 \end{eqnarray} 
Now from (\ref{2024-02-06}), (\ref{2024-02-07+1})
and 
 (\ref{2023-07-25+1})
we obtain 
\begin{eqnarray}
\nonumber
{\bf{P}}\{{\cal D}_{K(0)}(v)\cap {\cal D}_{K(k-1)}(u)\}
&=&
q_2^{m-2k+2}q_3^{2k-2}
\le
\bigl(1+o(1))\bigr)
e^{-2\kappa\frac{m}{n}}\left(\frac{\tau}{n}\right)^{2k-2},
\\
\nonumber
{\bf{P}}\{{\cal D}_{K(0)}(v)\cap {\cal D}_{K(h)}(u)\}
&=&
O\bigl(n^{1-k-h}\bigr)
e^{-2\kappa\frac{m}{n}},
\qquad
h=0,1,\dots, k-2.
\end{eqnarray}
Invoking these bounds in  
(\ref{2023-07-25+2}) and using 
$\frac{m^h}{n^h}=o\left(\frac{m^{k-1}}{n^{k-1}}\right)$
for $0<h\le k-2$  we have that
\begin{eqnarray}
\nonumber
{\bf{E}} \bigl({\bf I}_{{\cal D}(v)}{\bf I}_{{\cal D}(u)}\bigr)
&\le&
\left(1+o(1)\right)
{m\choose{k-1}}
{{m-k+1}\choose {k-1}}
e^{-2\kappa\frac{m}{n}}\left(\frac{\tau}{n}\right)^{2k-2}
%
%
%\\
%&=&
%\frac{m^{2(k-1)}}{((k-1)!)^2}\frac{\tau^{2(k-1)}}{n^{2(k-1)}}e^{-2\kappa\frac{m}{n}}
\\
\label{2023-07-25+3}
&=&
\left(1+o(1)\right)\left({\bf{P}}\{{\cal D}(v)\}\right)^2.
\end{eqnarray}
In the last step we used $\tau=(1+o(1))\tau^*$ 
and  (\ref{2024-02-02+1}).
It follows from (\ref{2023-07-25+3}) that
\begin{eqnarray}
\nonumber
{\bf Var} N'_{k-1}
&=&
{\bf{E}} (N'_{k-1})^2-({\bf{E}} N'_{k-1})^2
\\
\nonumber
&
=
&
n{\bf{P}}\{{\cal D}(v)\}
+
(n)_2
{\bf{E}} \bigl({\bf I}_{{\cal D}(v)}{\bf I}_{{\cal D}(u)}\bigr)
-
\left(n{\bf{P}}\{{\cal D}(v)\}\right)^2
\\
\nonumber
&
\le
&
n{\bf{P}}\{{\cal D}(v)\}
+
o(1)
\left(n{\bf{P}}\{{\cal D}(v)\}\right)^2
\\
\nonumber
&
=
&
{\bf{E}} N'_{k-1} +o({\bf{E}} N'_{k-1})^2.
\end{eqnarray}
In the case where ${\bf{E}} N'_{k-1}\to+\infty$ we obtain 
${\bf Var} N'_{k-1}=o({\bf{E}} N'_{k-1})^2$ as $n\to+\infty$. Now Chebyshev's inequality shows for any $\varepsilon>0$
\begin{displaymath}
{\bf{P}}\{|N'_{k-1}-{\bf{E}} N'_{k-1}|>\varepsilon {\bf{E}} N'_{k-1}\}
\le 
\frac{{\bf Var} N'_{k-1}}{(\varepsilon {\bf{E}} N'_{k-1})^2}
=
\frac{o(1)}{\varepsilon^2}
=
o(1).
\end{displaymath}
Hence $N'_{k-1}=(1+o_P(1)){\bf{E}} N'_{k-1}$.

It remains to show  (\ref{2024-02-06}).
Let
${\bf{P}}_X$  denote
 the conditional probability given $X_1$.
Let us estimate $q_1$. 
Identities $d_1(v)=1$,  $d_1(u)=1$ imply   
$\{u,v\}\subset {\cal V}_1$. In particular, we have $q_1\le {\bf{P}}\{\{u,v\}\subset {\cal V}_1\}=(n)_2^{-1}{\bf{E}}({\tilde X}_1)_2$. Furthermore,  we have
\begin{eqnarray}
\nonumber
q_1
&
=
&
{\bf{E}}{\bf{P}}_X\{d_1(v)=d_1(v)=1,\{u,v\}\subset {\cal V}_1\}
\\
\nonumber
&
=
&
{\bf{E}}\left(
{\bf{P}}_X\{d_1(v)=d_1(v)=1|\,\{u,v\}\subset {\cal V}_1\}
{\bf{P}}_X\{\{u,v\}\subset {\cal V}_1\}
\right)
\\
\nonumber
&
=
&
{\bf{E}}\left(
\left(
Q_1(1-Q_1)^{2({\tilde X}_1-2)}
+
(1-Q_1)\bigl(({\tilde X}_1-2)Q_1(1-Q_1)^{{\tilde X}_1-3}\bigr)^2\right)\frac{({\tilde X}_1)_2}{(n)_2}\right).
\end{eqnarray}
Here the first term $Q_1(1-Q_1)^{2({\tilde X}_1-2)}$ refers to the event $\{u,v\}\in {\cal E}_1$. The second term refers to the complement event $\{u,v\}\not\in {\cal E}_1$.

Next for $q\in [0,1]$  and $x=2,3,\dots$ we apply  inequalities
$(x-2)q(1-q)^{x-3}\le 1$ and  $1-q\le 1$ and derive the inequality
\begin{displaymath}
q(1-q)^{2(x-2)}
+
(1-q)\bigl((x-2)q(1-q)^{x-3}\bigr)^2
\le
q+(x-2)q=(x-1)q.
\end{displaymath}
Invoking this inequality in the  formula for $q_1$ above we obtain
\begin{displaymath}
q_1\le (n)_2^{-1}{\bf{E}}\left(({\tilde X}_1-1)({\tilde X}_1)_2Q_1\right)
\le n^{-1}{\bf{E}}\left({\tilde X}_1^2Q_1\right).
\end{displaymath} 
%{\color{red}
%Combining the two upper bounds for $q_1$ we show that $q_1\le %n^{-2}{\bf{E}}\left({\tilde X}_1^2\min\{1,{\tilde X}_1Q_1\}\right)$.
%}

Let us evaluate $q_2$.
We split
\begin{equation}
\label{2024-02-07}
q_2
= 
q_{2,1}+q_{2,2}+q_{2,3}+q_{2,4},
\end{equation}
where
\begin{eqnarray}
\nonumber
q_{2,1}
&
=
&
{\bf{P}}\{\{u,v\}\cap {\cal V}_1=\emptyset\},
\qquad
q_{2,2}={\bf{P}}\{\{u,v\}\subset {\cal V}_1, d_1(v)=d_1(u)=0\},
\\
\nonumber
q_{2,3}
&
=
&
{\bf{P}}\{v\in{\cal V}_1, u \notin{\cal V}_1, d_1(v)=0\},
\qquad
q_{2,4}
=
{\bf{P}}\{v\notin{\cal V}_1, u \in{\cal V}_1, d_1(u)=0\},
\end{eqnarray}
and calculate the probabilities
\begin{eqnarray}
\nonumber
q_{2,1}
&
=
&
1-{\bf{P}}\bigl\{\{u\in {\cal V}_1\}\cup\{v\in{\cal V}_1\}\bigr\}
%\\
%&
=
1
-
{\bf{P}}\{u\in {\cal V}_1\}
-
{\bf{P}}\{v\in{\cal V}_1\}
+ 
{\bf{P}}\{\{u,v\}\subset{\cal V}_1\}
\\
\nonumber
&=&
1-2\frac{{\bf{E}} {\tilde X}_1}{n}+\frac{{\bf{E}} ({\tilde X}_1)_2}{(n)_2},
\\
\nonumber
q_{2,2}
&
=
&
{\bf{E}} {\bf{P}}_{X}\{\{u,v\}\subset {\cal V}_1, d_1(v)=d_1(u)=0\}
\\
&
=
&
{\bf{E}}\bigl(
{\bf{P}}_{X}\{ d_1(v)=d_1(u)=0|\{u,v\}\subset {\cal V}_1\} 
{\bf{P}}_{X}\{\{u,v\}\subset {\cal V}_1\}
\bigr)
\\
&
=
&
{\bf{E}}\left((1-Q_1)^{2{\tilde X}_1-3}\frac{({\tilde X}_1)_2}{(n)_2}\right),
\\
\nonumber
q_{2,3}
=
q_{2,4}
&
=
&
{\bf{E}}
{\bf{P}}_X\{v\not\in{\cal V}_1, u\in {\cal V}_1, d_1(u)=0\}
\\
\nonumber
&
=
&
{\bf{E}}
\left({\bf{P}}_X\{ d_1(u)=0|v\not\in{\cal V}_1, u\in {\cal V}_1\}
{\bf{P}}_X\{v\not\in{\cal V}_1, u\in {\cal V}_1\}
\right)
\\
\nonumber
&
=
&
{\bf{E}}\left((1-Q_1)^{{\tilde X}_1-1}\left(1-\frac{{\tilde X}_1}{n}\right)\frac{{\tilde X}_1}{n-1}\right)
\\
\nonumber
&
=
&
\frac{1}{n}{\bf{E}}\left({\tilde X}_1(1-Q_1)^{{\tilde X}_1-1}\right)
-\frac{1}{(n)_2}{\bf{E}}\left(({\tilde X}_1)_2(1-Q_1)^{{\tilde X}_1-1}\right).
\end{eqnarray}
Invoking these expressions for $q_{2,1}, q_{2,2}, q_{2,3}, q_{2,4}$ in (\ref{2024-02-07}) we obtain 
$q_2
=
1
-
2
\frac{\kappa}{n}+\frac{\Delta_2}{(n)_2}$, where
\begin{eqnarray}
\nonumber
\Delta_2
&&
:=
{\bf{E}}
\left(
({\tilde X}_1)_2
\left(
1-2(1-Q_1)^{{\tilde X}_1-1}+(1-Q_1)^{2{\tilde X}_1-3}
\right)
\right)
\\
\nonumber
&&
\
\le
{\bf{E}}
\left(
({\tilde X}_1)_2
\left(
1-(1-Q_1)^{{\tilde X}_1-1}
\right)
\right)
.
\end{eqnarray}

Let us evaluate $q_3$. We split $q_3=q_{3,1}+q_{3,2}$, where
\begin{displaymath}
q_{3,1}
=
{\bf{P}}\{d_1(v)=1, u\notin{\cal V}_1\},
\qquad
q_{3,2}
=
{\bf{P}}\{d_1(v)=1, u\in {\cal V}_1, d_1(u)=0\},
\end{displaymath}
and calculate the probabilities
\begin{eqnarray}
\nonumber
q_{3,1}
&
=
&
{\bf{E}}
{\bf{P}}_X\{d_1(v)=1, u\not\in{\cal V}_1, v\in {\cal V}_1\}
\\
\nonumber
&
=
&
{\bf{E}}
\left({\bf{P}}_X\{ d_1(v)=1|u\not\in{\cal V}_1, v\in {\cal V}_1\}
{\bf{P}}_X\{u\not\in{\cal V}_1, v\in {\cal V}_1\}
\right)
\\
\nonumber
&
=
&
{\bf{E}}\left(Q_1(1-Q_1)^{{\tilde X}_1-2}({\tilde X}_1-1)\left(1-\frac{{\tilde X}_1}{n}\right)\frac{{\tilde X}_1}{n-1}\right)
\\
\nonumber
&
=
&
{\bf{E}}\left(Q_1(1-Q_1)^{{\tilde X}_1-2}({\tilde X}_1-1)\left(\frac{{\tilde X}_1}{n}-\frac{({\tilde X}_1)_2}{(n)_2}\right)\right)
\end{eqnarray}
and
\begin{eqnarray}
\nonumber
q_{3,2}
&
=
&
{\bf{E}}
{\bf{P}}_X\{d_1(v)=1,d_1(u)=0,  \{u,v\}\subset {\cal V}_1\}
\\
\nonumber
&
=
&
{\bf{E}}
\left(
{\bf{P}}_X
\{
d_1(v)=1,d_1(u)=0|\{u,v\}\subset {\cal V}_1
\}
{\bf{P}}_X\{\{u,v\}\subset{\cal V}_1\}
\right)
\\
\nonumber
&
=
&
{\bf{E}}
\left(
Q_1(1-Q_1)^{2{\tilde X}_1-4}({\tilde X}_1-2)
\frac{({\tilde X}_1)_2}{(n)_2}
\right).
\end{eqnarray}
We obtain 
\begin{displaymath}
q_3
=
 {\bf{E}}\left(\frac{({\tilde X}_1)_2}{n} 
Q_1(1-Q_1)^{{\tilde X}_1-2}\right)
-
{\bf{E}}\left(\frac{({\tilde X}_1)_2}{(n)_2}\theta\right)
=
\frac{\tau}{n}
-
{\bf{E}}\left(\frac{({\tilde X}_1)_2}{(n)_2}\theta\right),
\end{displaymath}
where $0\le \theta\le 1$ stands for the difference of two  probabilities
\begin{eqnarray}
\nonumber
\qquad
\theta
&
=
&
Q_1(1-Q_1)^{{\tilde X}_1-2}({\tilde X}_1-1)
-
Q_1(1-Q_1)^{2{\tilde X}_1-4}({\tilde X}_1-2)
\\
\nonumber
&
=
&
{\bf{P}}_X\{d_1(v)=1\}
-
{\bf{P}}_X\{d_1(v)=1,d_1(u)=0\}
\\
\nonumber
&
\le
&
{\bf{P}}_X\{d_1(v)=1\}
=Q_1(1-Q_1)^{{\tilde X}_1-2}({\tilde X}_1-1).
\qquad
\quad
\qed
\end{eqnarray}

Now we are ready to prove Lemma \ref{degree}.

 {\it Proof of Lemma \ref{degree}}.
For $\lambda_{n,m,k}\to+\infty$ Lemmas  
\ref{approx}  
and
\ref{concentration} 
imply   ${\bf{E}} N'_{k-1}\to+\infty$ and 
\begin{displaymath}
N_{k-1}=(1+o_P(1))N'_{k-1}=(1+o_P(1)){\bf{E}} N'_{k-1}
\end{displaymath}
 Hence ${\bf{P}}\{N_{k-1}\ge 1\}\to 1$. 

Assume now that $\lambda_{n,m,k}\to- \infty$. Then 
$\lambda_{n,m,t}\to- \infty$ for $t=1,\dots, k$. 
For $h=1,\dots, k-1$ relations  (\ref{2024-02-01}), (\ref{2024-02-02}) of
Lemma \ref{approx} imply that
\begin{displaymath}
{\bf{E}} N'_{h}=o(1),
\qquad
{\bf{E}}|N_{h}-N'_{h}|=o(1).
\end{displaymath}
We have
$|{\bf{E}} N_h|\le{\bf{E}}|N_{h}-N'_{h}|+ |{\bf{E}} N'_{h}| =o(1)$
and
hence $N_h=o_P(1)$.

 For $h=0$ the bound $N_h=o_P(1)$ follows from the fact that $\lambda_{n,m,1}\to- \infty$ implies 
that $G_{[n,m]}$ is connected with high probability, see \cite{Daumilas_Mindaugas2023}.
 $\qed$


\begin{thebibliography}{24}




\bibitem{Daumilas_Mindaugas2023}
Ardickas, D., Bloznelis, M.:  
Connectivity threshold for superpositions of Bernoulli random graphs.
arxiv:2306.08113v2 (2023)


\bibitem{BergmanLeskela}
 Bergman, E.,  Leskelä, L.: 
Connectivity of random hypergraphs with a given hyperedge size
distribution. Discrete Applied Mathematics  \textbf{357}(15), 1--13 (2024)

\bibitem{Bloznelis2013AAP}
Bloznelis, M.:
Degree and clustering coefficient in sparse random intersection graphs, 
{\em The Annals of Applied Probability} 23, 1254--1289  (2013)

%\bibitem{Joona_Lasse_Mindaugas2021}
%Bloznelis, M.,  Karjalainen J., and  Leskel\"a, L.:
%Assortativity and bidegree distributions on Bernoulli random graph 
%superpositions. Probability in the Engineering and Informational Sciences 
%\textbf{36}(4), 1188--1213 (2021)

\bibitem{Joona_Lasse_Mindaugas2024}
Bloznelis, M.,  Karjalainen J., and  Leskel\"a, L.: 
Normal and stable approximation to subgraph counts in superpositions of 
Bernoulli random graphs.   Journal of Applied Probability \textbf{61}, 
401--419 (2024) 


\bibitem{Lasse_Mindaugas2019}
Bloznelis, M., Leskel\"a, L.:
Clustering and percolation on superpositions of Bernoulli random graphs.
 Random Structures $\&$ Algorithms \textbf{63}(2), 283--342 (2023)


\bibitem{Dominykas_Mindaugas_Rimantas2023}
Bloznelis, M., Marma, D., Vaicekauskas, R.:
Connectivity threshold for superpositions of Bernoulli random graphs.
II.
arXiv:2311.09317   (2023)


\bibitem{Bloznelis_Rybarczyk_2014}
Bloznelis, M., Rybarczyk, K.:  k-connectivity of uniform s-intersection graphs. Discrete Mathematics \textbf{333}, 94--100 (2014)



\bibitem{Bollobas_1981}
Bollob\'as, B.: Random graphs. In Combinatorics (ed. H. N. V. Temperley), London Mathematical Society Lecture Note Series, \textbf{52}, Cambridge
University Press, Cambridge 80–102 (1981)



\bibitem{Devroye_Fraiman_2014}
Devroye, L., Fraiman, N.: Connectivity of inhomogeneous random graphs. Random Structures Algorithms \textbf{45}, 408--420 (2014)

\bibitem{ErdosRenyi1959}
Erd\H os, P.,  R\'enyi, A.:
On random graphs. I.
 Publ. Math. Debrecen \textbf{6}, 290--297  (1959)


\bibitem{ErdosRenyi1961}
Erd\H os, P.,  R\'enyi, A.:
 On the strength of connectedness of a random graph,
Acta Mathematica Hungarica \textbf{12}, 261--267 (1961)

\bibitem{FriezeKaronski}
Frieze, A., Karo\'nski, M.: 
Introduction to Random Graphs. Cambridge University Press, Cambridge
 (2015). 


\bibitem{GodehardJaworskiRybarczyk2007}
Godehardt,  E., Jaworski, J., Rybarczyk, K.: 
Random intersection graphs and
classification. In:  Decker, R., Lenz,  H.J. (eds.) Advances in Data Analysis,  pp. 67--74.  Springer, Berlin, Heidelberg 
(2007)

\bibitem{Grohn_Karjalainen_Leskela_2024}
 Gröhn, T., Karjalainen, J., Leskelä, L.:
Clique and cycle frequencies in a sparse random graph model with overlapping communities.  Stohastic models
 (2024)
 
 
 
\bibitem{Hofstad2017}
 van der Hofstad, R.: Random graphs and complex networks. Vol. 1. Cambridge Series in Statistical and Probabilistic Mathematics.
  Cambridge University Press, Cambridge
 (2017)
 
 \bibitem{Vadon_Komjathy_Hofstad_2021}  
van der Hofstad, R., Komj\'{a}thy, J., Vadon, V.:  Random intersection graphs with communities.
Adv. in Appl. Probab. \textbf{53}, 1061--1089 (2021)
 
\bibitem{JansonLuczakRucinski2001}
 Janson, S., \L uczak, T., Ruci\'{n}ski, A.:
Random Graphs.
Wiley,  New York (2000)


\bibitem{Menger_1927}
Menger, K.: Zur allgemeinen Kurventheorie.
 Fund. Math. \textbf{10}, 96--115 (1927)

\bibitem{Penrose_1999}
Penrose, M. D.:
On  $k$-connectivity for a geometric random graph.
Random Structures Algorithms. \textbf{15},  145--164 (1999)
   
\bibitem{Petti_Vempala_2022}
Petti, S., Vempala, S.S.: Approximating sparse graphs: The random overlapping communities model.
 Random Structures $\&$ Algorithms \textbf{61}, 844--908 (2022)

   \bibitem{Shang_2023}
   Shang, Y.: On connectivity and robustness of random graphs with inhomogeneity.
J. Appl. Probab.\textbf{60}, 284--294  (2023)
  
  \bibitem{Wormald_1981}
  Wormald, N. C.: The asymptotic connectivity of labelled regular graphs.
Journal of Combinatorial Theory. Series B \textbf{31}, 156–167 (1981)

\bibitem{Yang_Leskovec2012}
Yang, J., Leskovec, J.:
Community-affiliation graph model for overlapping network community detection. In  IEEE 12th International Conference on Data Mining, 
pages 1170--1175. IEEE (2012)

\bibitem{Yang_Leskovec2014}
Yang, J., Leskovec, J.: Structure and overlaps of ground-truth communities in networks.  ACM Trans. Intell. Syst. Technol.  
\textbf{5}(2), 1--35  (2014)
   
   \bibitem{Zhao_yagan_Gligor_2017}
   Zhao, J., Yağan, O., Gligor, V.:
On connectivity and robustness in random intersection graphs.
IEEE Trans. Automat. Control \textbf{62}, 2121--2136 (2017)
   
   \end{thebibliography}
\end{document}